\documentclass[a4paper,11pt]{amsart}
\usepackage[left=2.7cm,right=2.7cm,top=3.5cm,bottom=3cm]{geometry}
\usepackage{amssymb,latexsym,amsmath,amsthm,amscd}
\usepackage{graphicx}
\usepackage{dsfont}
\usepackage[all]{xy}
\usepackage{mathrsfs}
\usepackage{color}
\definecolor{Blue}{rgb}{0.3,0.3,0.9}

\usepackage[T2A,OT1]{fontenc}
\DeclareSymbolFont{cyrillic}{T2A}{cmr}{m}{n}
\DeclareMathSymbol{\Sha}{\mathalpha}{cyrillic}{216}

\newcommand{\sk}{\vspace{0.1in}}
\newtheorem{thm}{Theorem}[section]
\newtheorem{def-thm}[thm]{Definition-Theorem}
\newtheorem{cor}[thm]{Corollary}
\newtheorem{lem}[thm]{Lemma}
\newtheorem{def-lem}[thm]{Definition-Lemma}
\newtheorem{prop}[thm]{Proposition}
\newtheorem{conj}[thm]{Conjecture}
\newtheorem*{ThmA}{Theorem A}
\newtheorem*{ThmB}{Theorem B}
\newtheorem*{ThmC}{Theorem C}

\theoremstyle{definition}
\newtheorem{defn}[thm]{Definition}
\theoremstyle{remark}

\newtheorem{rem}[thm]{Remark}
\numberwithin{thm}{section}
\numberwithin{equation}{section}

\newcommand{\wh}[1]{\hat#1}

\newcommand{\A}{\mathcal{A}}

\newcommand{\cO}{\mathcal{O}}

\newcommand{\frakm}{\mathfrak{m}}

\newcommand{\frakl}{\mathfrak{l}}
\newcommand{\pp}{\mathfrak{p}}
\newcommand{\frakp}{\mathfrak{p}}
\newcommand{\qq}{\mathfrak{q}}
\newcommand{\frakq}{\mathfrak{q}}

\newcommand{\frakf}{}

\newcommand{\fraka}{\mathfrak{a}}
\newcommand{\frakc}{\mathfrak{c}}

\newcommand{\cI}{\mathcal{I}}
\newcommand{\bQ}{\mathbf{Q}}
\newcommand{\bZ}{\mathbf{Z}}
\newcommand{\bC}{\mathbf{C}}

\newcommand{\Ac}{{\mathbf{A}^{\rm ac}}}
\newcommand{\Tc}{{\mathbf{T}^{\rm ac}}}

\newcommand{\T}{\mathbf{T}^\dagger}

\newcommand{\ro}{\mathbf{Z}_p}
\newcommand{\cyc}{\rm cyc}

\newcommand{\can}{\Psi}
\newcommand{\unr}{R_0}
\newcommand{\eps}{\varepsilon}

\newcommand{\tors}{\rm tors}

\def\ac{{\rm ac}}
\def\BF{{\mathcal{BF}}}

\newcommand{\newabstract}[1]{%
	\par\bigskip
	\csname otherlanguage*\endcsname{#1}%
	\csname captions#1\endcsname
	\item[\hskip\labelsep\scshape\abstractname.]
}

\usepackage[french,english]{babel}

\begin{document}

\title[Perrin-Riou's main conjecture for elliptic curves at supersingular primes]{Perrin-Riou's 
	main conjecture for elliptic curves at supersingular primes}
\author[F.~Castella and X.~Wan]{Francesc Castella and Xin Wan}

\address{Department of Mathematics, Princeton University, Fine Hall, 
	Princeton, NJ 08544-1000, USA}
\email{fcabello@math.princeton.edu}
\address{Morningside Center of Mathematics, Academy of Mathematics and Systems Science, Chinese Academy of Science, No.~55 Zhongguancun East Road, Beijing, 100190, China}
\email{xwan@math.ac.cn}
\thanks{
	This project has received funding from the European Research Council (ERC) under the European Union's
	Horizon 2020 research and innovation programme (grant agreement No. 682152).}



\begin{abstract}
In 1987, B.~Perrin-Riou formulated a Heegner point main conjecture for elliptic curves at primes of ordinary reduction. In this paper, we formulate an analogue of Perrin-Riou's main conjecture for supersingular primes. We then prove this conjecture under mild hypotheses, and deduce from this result a $\Lambda$-adic extension of Kobayashi's $p$-adic Gross--Zagier formula, new cases of B.-D.~Kim's doubly-signed main conjectures,  and a strengthened version of Skinner's converse to the Gross--Zagier--Kolyvagin theorem for supersingular primes.
\end{abstract}

\selectlanguage{english}


\maketitle

\setcounter{tocdepth}{2}
\tableofcontents

\section{Introduction}


\subsection{Perrin-Riou's main conjecture} 

Let $E/\bQ$ be an elliptic curve of conductor of $N$, 
let 
$f$ be the associated newform on $\Gamma_0(N)$, and fix an odd prime $p\nmid N$. Let $K/\bQ$ be an imaginary quadratic field of discriminant prime to $N$. Writing 
\[
N=N^+N^-
\] 
with $N^+$ (resp. $N^-$) only divisible by primes which split (resp. are inert) in $K$, assume that the following \emph{generalized Heegner hypothesis} is satisfied:
\begin{equation}\label{eq:gen-HH}
\textrm{$N^-$ is the square-free product of an even number of primes.}
\tag{Heeg}
\end{equation}

Assuming that $p$ is \emph{ordinary} for $E$, Perrin-Riou formulated in \cite{PR-HP} an Iwasawa-theoretic main conjecture for Heegner points over the anticyclotomic tower. To recall its statement, let $K_\infty^{\rm ac}=\bigcup_{n}K_n^{\ac}$ be the anticyclotomic $\bZ_p$-extension of $K$, with $K_n^{\ac}$ the unique subextension of $K_\infty^{\ac}$ of degree $p^n$ over $K$, and set
\begin{equation}\label{eq:bigSel}
{\rm Sel}_{p^\infty}(E/K_\infty^\ac):=\varinjlim_n\varinjlim_m{\rm Sel}_{p^m}(E/K_{n}^{\ac}),\quad
{\rm Sel}(K_\infty^{\ac},T_pE):=\varprojlim_n\varprojlim_m{\rm Sel}_{p^m}(E/K_n^{\ac}),
\end{equation}
where
\begin{equation}\label{eq:local-Sel}
{\rm Sel}_{p^m}(E/K_n^{\ac}):={\rm ker}\biggl\{H^1(K_n^{\ac},E[p^m])\longrightarrow\prod_{w}H^1(K_{n,w}^\ac,E)\biggr\}
\end{equation}
is the $p^m$-descent Selmer group of $E$ over $K_n^\ac$. 
Another assumption in \cite{PR-HP} is that  $N^-=1$. Taking Heegner points on $X_0(N)$ and mapping them to $E$ by a fixed modular parametrization
\[
\varphi:X_0(N)\longrightarrow E
\]
one can construct a $\Lambda_{\rm ac}$-module $H_\infty\subset {\rm  Sel}(K_\infty^\ac,T_pE)$, where
\[ 
\Lambda_{\ac}:=\bZ_p[[{\rm Gal}(K_\infty^\ac/K)]]
\] 
is the anticyclotomic Iwasawa algebra. After the work of Cornut \cite{cornut} and Vatsal \cite{vatsal}, the module $H_\infty$ is known to be free of rank $1$, say $H_\infty=\Lambda_{\ac}\cdot\mathbf{z}_\infty$.
\sk 

Let $c_E\in\bQ^\times$ be the Manin constant associated with $\varphi$ (i.e., if $\omega$ is a N\'eron differential of $E$ and $f$ is the newform associated with $E$, then $\varphi^*\omega=c_E\cdot 2\pi if(z)dz$), and let $u_K:=\vert\cO_K^\times\vert/2$.  

\begin{conj}[Perrin-Riou]\label{conj:PR-ord}
	Assume that $N^-=1$ and that $p$ is a prime of good ordinary reduction of $E$. Then the Pontrjagin dual $X_{p^\infty}(E/K_{\infty}^\ac)$ of ${\rm Sel}_{p^\infty}(E/K_{\infty}^\ac)$ has $\Lambda_\ac$-rank $1$, and
	\[
	{\rm Char}_{\Lambda_\ac}(X_{p^\infty}(E/K_\infty^\ac)_{\rm tors})=\frac{1}{c_E^2u_K^2}\cdot{\rm Char}_{\Lambda_\ac}\biggl(\frac{{\rm Sel}(K_\infty^\ac,T_pE)}{\Lambda_\ac\cdot\mathbf{z}_\infty}\biggr)^2,
	\]
	where the subscript {\rm tors} denotes the $\Lambda^\ac$-torsion submodule.
\end{conj}

The important works of Bertolini \cite{bertolini-PhD} and Howard \cite{howard-PhD-I} led to the proof (under mild hypotheses) of one of the divisibilities predicted by Conjecture~\ref{conj:PR-ord}, while later in \cite{howard-PhD-II}  Howard formulated an extension of Conjecture~\ref{conj:PR-ord} 
to abelian varieties of ${\rm GL}_2$-type over totally real fields, and extended the results of \cite{howard-PhD-I} to this context. 
\sk

Our first main goal in this paper is to formulate an extension of Conjecture~\ref{conj:PR-ord} to good supersingular primes, working under the generalized Heegner hypothesis (\ref{eq:gen-HH}). A fundamental obstacle to such extension is the fact that, in the non-ordinary case, i.e., when $p$ divides
\[
a_p:=p+1-\vert E(\mathbf{F}_p)\vert, 
\]
Heegner points on $E$ give rise to compatible systems of classes with \emph{unbounded} growth over the anticyclotomic tower. As a result, there is no obvious analogue of the submodule $H_\infty$ of ${\rm Sel}(K_\infty^\ac,T_pE)$ for supersingular primes $p$. Nonetheless, following ideas of Kobayashi \cite{kobayashi-152}, and their extension by B.-D.~Kim \cite{kim-CJM}, we will succeed in formulating the right analogue of Conjecture~\ref{conj:PR-ord} assuming in addition\footnote{Note that, by the Hasse bound, the condition $a_p=0$ holds for all supersingular primes $p>3$.} that $a_p=0$  and that
\begin{equation}\label{eq:split}
\textrm{$p=\pp\overline{\pp}$ splits in $K$.}\tag{spl}
\end{equation}
 
Indeed, following these methods, in Section~\ref{sec:Sel} we define four pairs of doubly-signed Selmer groups 
\[
\mathfrak{Sel}_{p^\infty}^{\pm,\pm}(E/K_\infty^\ac)\subset
{\rm Sel}_{p^\infty}(E/K_\infty^\ac),\quad
\mathfrak{Sel}^{\pm,\pm}(K^\ac_\infty,T_pE)\supset{\rm Sel}(K^\ac_\infty,T_pE)
\]
obtained by replacing the local conditions
in $(\ref{eq:local-Sel})$ at the primes above $\pp$ and $\overline{\pp}$. On the other hand, in Section~\ref{sec:HP} we construct two bounded cohomology classes
\[
\mathbf{z}^+\in\mathfrak{Sel}^{+,+}(K_\infty^\ac,T_pE),\quad
\mathbf{z}^-\in\mathfrak{Sel}^{-,-}(K_\infty^\ac,T_pE).
\]
obtained by dividing the natural systems of Heegner points over $K_\infty^\ac/K$ by certain analogues of Pollack's \cite{pollack} half-logarithms. 
By (\ref{eq:gen-HH}), the curve $E$ is isogenous to a quotient
\begin{equation}\label{eq:Sh-par}
\pi:J_{N^+,N^-}\longrightarrow E'
\end{equation}
of the Jacobian of a Shimura curve $X_{N^+,N^-}$ attached to an indefinite quaternion algebra over $\bQ$ of discriminant $N^-$, and the Heegner points used to construct $\mathbf{z}^\pm$ come from $X_{N^+,N^-}$. 
Let $\delta(N^+,N^-)=\pi\circ\pi^\vee$ be the modular degree of the parametrization (\ref{eq:Sh-par}),  and assuming that $E$ is the strong Weil curve in its isogeny class, set
\[
\delta_{N^+,N^-}:=\frac{\delta(N^+,N^-)}{\delta(N,1)},
\]
where $\delta(N,1)=\varphi\circ\varphi^\vee$ is the modular degree of $\varphi$. We are now ready to state our 
generalization of Conjecture~\ref{conj:PR-ord}.

\begin{conj}\label{conj:PR-ss}
Assume the generalized Heegner hypothesis {\rm (\ref{eq:gen-HH})} holds, and let $p>3$ be a prime of good supersingular reduction of $E$ which splits in $K$. Then, for each $\varepsilon\in\{\pm\}$, the Pontrjagin dual $\mathfrak{X}_{p^\infty}^{\varepsilon,\varepsilon}(E/K_{\infty}^\ac)$ of $\mathfrak{Sel}^{\varepsilon,\varepsilon}_{p^\infty}(E/K_{\infty}^\ac)$ has $\Lambda_\ac$-rank $1$, and
\[
{\rm Char}_{\Lambda_\ac}(\mathfrak{X}_{p^\infty}^{\varepsilon,\varepsilon}(E/K_\infty^\ac)_{\rm tors})=\frac{\delta_{N^+,N^-}}{c_E^2u_K^2}\cdot{\rm Char}_{\Lambda_\ac}\biggl(\frac{\mathfrak{Sel}^{\varepsilon,\varepsilon}(K_\infty^\ac,T_pE)}{\Lambda_\ac\cdot\mathbf{z}^\varepsilon_\infty}\biggr)^2,
\]
where the subscript {\rm tors} denotes the $\Lambda_\ac$-torsion submodule.
\end{conj}

\subsection{Statement of the main results} As mentioned above, one of the divisibilities in Perrin-Riou's main conjecture follows from \cite{bertolini-PhD} and \cite{howard-PhD-I}. More recently, the authors established the converse divisibility, leading to a proof of Conjecture~\ref{conj:PR-ord} under mild hypotheses (see \cite{wan} and \cite{cas-BF}). As for the supersingular setting, 
our first main result in this paper is the following result on Conjecture~\ref{conj:PR-ss}:

\begin{ThmA}
Assume that: 
\begin{itemize}
	\item $N$ is square-free,
	\item $N^-\neq 1$,
	\item $E[p]$ is ramified at every prime $\ell\mid N^-$.
\end{itemize} 
Then Conjecture~\ref{conj:PR-ss} holds.	
\end{ThmA}

Similarly as in \cite{wan}, Theorem~A yields in particular a converse to the Gross--Zagier--Kolyvagin theorem, 
in the same spirit as the result first obtained by Skinner 
\cite[Thm.~B]{skinner} for ordinary primes (\emph{cf.} W.~Zhang's result  \cite[Thm.~1.3]{zhang-Kolyvagin}, still for ordinary primes). Both  approaches make crucial use of Iwasawa theory, but by working with Heegner
points over the tower $K_\infty^\ac/K$, our result does \emph{not} require any
injectivity hypothesis on the localization maps at places above $p$:
\[
{\rm loc}_p:{\rm Sel}(K,V_pE)\longrightarrow\prod_{w\mid p}H^1(K_w,V_pE),
\]
thereby dispensing with the finiteness of the $p$-primary part of 
$\Sha(E/K)$ 
(see \cite[Lem.~2.2.2]{skinner}), and rather deducing it as a consequence.


\begin{ThmB}
	Under the hypotheses in Theorem~A, the following implication holds: 
	\[
	{\rm corank}_{\bZ_p}{\rm Sel}_{p^\infty}(E/K)=1\quad\Longrightarrow\quad{\rm ord}_{s=1}L(E/K,s)=1.
	\]
\end{ThmB}

In the course of proving Theorem~A, we also obtain new cases of B.-D.~Kim's doubly-signed main conjectures for elliptic curves at  supersingular primes \cite{kim-CJM}. For the statement, fix a root $\alpha$ of $x^2-a_px+p=x^2+p$, let $\beta=-\alpha$ be the other root, and denote by $\Gamma_K:={\rm Gal}(K_\infty/K)$ the Galois group of the unique $\bZ_p^2$-extension of $K$. 
Building upon Haran's construction \cite{haran} of Mazur--Tate elements for automorphic forms on ${\rm GL}_{2}$ over number fields, 
Loeffler introduced in \cite{Loeffler}:
\begin{itemize}
\item[-]{} Four unbounded distributions on $\Gamma_K$:
\begin{equation}\label{eq:4-Loeffler}
L_{p,(\alpha,\alpha)}(E/K),\quad
L_{p,(\alpha,\beta)}(E/K),\quad
L_{p,(\beta,\alpha)}(E/K),\quad
L_{p,(\beta,\beta)}(E/K),\quad
\end{equation}
interpolating the Rankin--Selberg $L$-values $L(E/K,\psi,1)$, as $\psi$ runs over
the finite order characters of $\Gamma_K$;
\item[-]{} Four bounded $\bQ_p$-valued measures on $\Gamma_K$:
\begin{equation}\label{eq:4-loeffler-b}
L_p^{+,+}(E/K),\quad
L_p^{-,+}(E/K),\quad
L_p^{+,-}(E/K),\quad
L_p^{-,-}(E/K),\quad
\end{equation}
for which one has the decomposition
\begin{equation}
\begin{aligned}\label{eq:loeffler}
L_{p,(\alpha,\beta)}(E/K)=&\;L_p^{+,+}(E/K)\cdot\log^+_{\pp}\log^+_{\overline\pp}\\
+&L_p^{-,+}(E/K)\cdot\log^+_{\pp}\log^-_{\overline\pp}\cdot\alpha
  +L_p^{+,-}(E/K)\cdot\log^-_{\pp}\log^+_{\overline\pp}\cdot\beta\\
+&L_p^{-,-}(E/K)\cdot\log^-_\pp\log_{\overline{\pp}}^-\cdot\alpha\beta,
\end{aligned}
\end{equation}
and similarly for the other three distributions in (\ref{eq:4-Loeffler}),
for certain elements $\log_\pp^{\pm}, \log^{\pm}_{\overline\pp}\in\bQ_p[[\Gamma_K]]$ generalizing Pollack's \cite{pollack} half-logarithms.
\end{itemize}

On the arithmetic side, B.-D.~Kim \cite{kim-CJM} introduced four doubly-signed Selmer groups ${\rm Sel}_{p^\infty}^{\pm,\pm}(E/K_\infty)$, which he conjectured to be cotorsion over the two-variable Iwasawa algebra $\Lambda_K:=\bZ_p[[\Gamma_K]]$, with characteristic ideal generated by 
$L_p^{\pm,\pm}(E/K)$ (\emph{cf.} \cite[Conj.~3.1]{kim-CJM}):

\begin{conj}[Kim]\label{conj:kim}
For each $\bullet, \circ=\{\pm\}$,
the Pontrjagin dual $X_{p^\infty}^{\bullet,\circ}(E/K_\infty)$ of
${\rm Sel}_{p^\infty}^{\bullet,\circ}(E/K_\infty)$ is $\Lambda_K$-torsion, and
\[
{\rm Char}_{\Lambda_K}(X_{p^\infty}^{\bullet,\circ}(E/K_\infty))=(L_p^{\bullet,\circ}(E/K))
\]
as ideals in $\Lambda_K$.
\end{conj}

Note that Conjecture~\ref{conj:kim} is the combination of four different main conjectures, which no direct connection \emph{a priori}  between them. 
In \cite{wan-combined}, the second named author has obtained (under mild hypotheses) the proof one of the divisibilities predicted by the two equal-sign cases of Conjecture~\ref{conj:kim} 
when the global root number of $E/K$ is $+1$ (so that $N^-$ is the product of an \emph{odd} number of primes). 
In this paper, we extend this result 
to the cases when the global root number of $E/K$ is $-1$: 


\begin{ThmC}
Under the hypotheses in Theorem~A, for each $\varepsilon\in\{\pm\}$ the module $X_{p^\infty}^{\varepsilon,\varepsilon}(E/K_\infty)$ is $\Lambda_K$-torsion, and
\[
{\rm Char}_{\Lambda_K}(X_{p^\infty}^{\varepsilon,\varepsilon}(E/K_\infty))=(L_p^{\varepsilon,\varepsilon}(E/K))
\]
as ideals in $\Lambda_K$.
\end{ThmC}

Finally, we note that our methods 
also yield a proof under mild hypotheses 
of the Iwasawa--Greenberg main conjecture for the $p$-adic $L$-functions of Bertolini--Darmon--Prasanna 
\cite{bdp1} (see Theorem~\ref{thm:howard}), as well as a $\Lambda_{\ac}$-adic extension of Kobayashi's $p$-adic Gross--Zagier formula \cite{kobayashi-191} (see Theorem~\ref{thm:lambda-GZ}) 
for elliptic curves at supersingular primes $p>3$.

\subsection{Outline of the proofs}

The proofs of our main results are via Iwasawa theory, exploiting the connections between Heegner points, Beilinson--Flach classes, and their explicit reciprocity laws. Recall that $\alpha$ and $\beta$ denote the roots of $x^2+p$, 
and let $f_\alpha$ and $f_\beta$ be the $p$-stabilizations of $f$
with $U_p$-eigenvalues $\alpha$ and $\beta$, respectively. In \cite{LZ-Coleman}, Loeffler and Zerbes have defined
three-variable systems of cohomology classes $\BF_{\mathbf{f},\mathbf{g}}$ interpolating the
Beilinson--Flach classes of \cite{LLZ} attached to the different specializations of two Coleman families
$\mathbf{f}$ and $\mathbf{g}$ and their cyclotomic twists. Letting $\mathbf{f}$ be the Coleman family passing through $f_\alpha$ and $f_\beta$, respectively,
and $\mathbf{g}$ be a certain Hida family of CM forms, 
we deduce from their work the construction of two-variable classes
\begin{equation}\label{eq:BF}
\BF_{\alpha},\;\BF_{\beta}\in\bQ_p[[\Gamma_K]]\otimes_{\Lambda_K}
H^1_{\rm Iw}(K_\infty,T_pE).
\end{equation}

In analogy with $(\ref{eq:4-Loeffler})$, the classes $(\ref{eq:BF})$ have unbounded growth over the tower $K_\infty/K$, but building on the explicit reciprocity laws of \cite{LZ-Coleman} we deduce from them the construction of two \emph{bounded} elements $\BF_{}^\pm\in H^1_{\rm Iw}(K_\infty,T_pE)$. Moreover, we construct four 
$\Lambda_K$-linear maps\footnote{The maps ${\rm Log}^\pm$ are in fact valued in a large scalar extension of $\Lambda_K$ that we suppress here for the ease of exposition.}
\begin{align*}
{\rm Col}^{\pm}:\frac{H_{\rm Iw}^1(K_{\infty,\overline{\pp}},T_pE)}{H_{\pm,{\rm Iw}}^1(K_{\infty,\overline{\pp}},T_pE)}
\longrightarrow\Lambda_K,
\quad\quad
{\rm Log}^{\pm}:H_{\pm,{\rm Iw}}^1(K_{\infty,\pp},T_pE)
\longrightarrow\Lambda_K,
\end{align*}
such that
\begin{equation}\label{eq:ERL-BF1}
{\rm Col}^{\circ}({\rm loc}_{\overline\pp}(\BF^\bullet))=L_p^{\bullet,\circ}(E/K)
\end{equation}
for all $\bullet,\circ\in\{\pm\}$, 
and
\begin{equation}\label{eq:ERL-BF2}
{\rm Log}^{\varepsilon}({\rm loc}_{\pp}(\BF^\varepsilon))=L_\pp(E/K)
\end{equation}
for all $\varepsilon\in\{\pm\}$, where $H^1_{\pm,{\rm Iw}}(K_{\infty,\overline{\pp}},T_pE)$ 
is the local condition defining ${\rm Sel}^{\pm,\pm}(K_\infty,T_pE)$ at the places above $\overline{\pp}$,
and similarly $H^1_{\pm,{\rm Iw}}(K_{\infty,\pp},T_pE)$ 
for the places above $\pp$, 
and $L_\pp(E/K)$ is a Rankin--Selberg $p$-adic 
$L$-function constructed in \cite{wan-combined}. 
\sk

The Iwasawa--Greenberg main conjecture \cite{Greenberg55} predicts that $L_\pp(E/K)$ generates the characteristic ideal of a certain torsion Selmer group: 
\begin{equation}\label{eq:Gr-conj}
{\rm Char}_{\Lambda_K}(\mathfrak{X}_{p^\infty}^{{\rm rel},{\rm str}}(E/K_\infty))\overset{?}=(L_\pp(E/K)),
\end{equation}
where $\mathfrak{X}_{p^\infty}^{{\rm rel},{\rm str}}(E/K_\infty)$ is the 
Pontrjagin dual of a Selmer group 
$\mathfrak{Sel}_{p^\infty}^{{\rm rel},{\rm str}}(f/K_\infty)$ 
defined by imposing local triviality (resp. no condition) at the places above $\overline{\pp}$ (resp. $\pp$). In \cite{wan-combined}, the second named author has obtained one of the divisibilities in conjecture $(\ref{eq:Gr-conj})$. 
By descending to the anticyclotomic line, we show that this leads to the divisibility
\begin{equation}\label{eq:div-wan}
{\rm Char}_{\Lambda_{\ac}}(\mathfrak{X}_{p^\infty}^{{\rm rel},{\rm str}}(E/K^{\ac}_\infty))\subseteq(\mathscr{L}_\pp^{\tt BDP}(E/K)^2)
\end{equation}
in the Iwasawa--Greenberg main conjecture for (the square of) the $p$-adic $L$-function of \cite{bdp1}. On the other hand, by an extension of Howard's techniques \cite{howard-PhD-I} to the Heegner classes $\mathbf{z}^\pm$, in Section~\ref{subsec:KS} we prove a divisibility in the opposite direction in Conjecture~\ref{conj:PR-ss}:
\begin{equation}\label{eq:div-how}
{\rm Char}_{\Lambda_\ac}(\mathfrak{X}_{p^\infty}^{\varepsilon,\varepsilon}(E/K_\infty^\ac)_{\rm tors})\supseteq{\rm Char}_{\Lambda_\ac}\biggl(\frac{\mathfrak{Sel}^{\varepsilon,\varepsilon}(K_\infty^\ac,T_pE)}{\Lambda_\ac\cdot\mathbf{z}^\varepsilon_\infty}\biggr)^2,
\end{equation}
By the explicit reciprocity law, analogous to $(\ref{eq:ERL-BF2})$, that we obtain in Theorem~\ref{3.1}:
\[
{\rm Log}_{\rm ac}^\varepsilon({\rm res}_\pp(\mathbf{z}^\varepsilon))=\mathscr{L}_\pp^{\tt BDP}(E/K)
\]
we show that the divisibilities $(\ref{eq:div-wan})$ and $(\ref{eq:div-how})$ complement each other, leading to the equalities in both. In particular, Theorem~A follows, and using the reciprocity laws $(\ref{eq:ERL-BF1})$ and $(\ref{eq:ERL-BF2})$ we deduce from this the proof of Theorem~C. 
\sk

We end this Introduction with a few remarks on related results in the literature, especially in the works of Longo--Vigni \cite{LV-ss} and  B{\"u}y{\"u}kboduk--Lei \cite{BuyLei-SS}. More precisely, one of the main results of \cite{LV-ss} amounts to the ``rank part'' of our Conjecture~\ref{conj:PR-ss}. The results of \cite{LV-ss} are based on an extension of Bertolini's techniques  \cite{bertolini-PhD} to supersingular primes, whereas here we deduce this portion of Theorem~A from an analogous extension of Howard's \cite{howard-PhD-I} (giving us access to the  $\Lambda_\ac$-torsion submodule of $X_{p^\infty}^{\varepsilon,\varepsilon}(E/K_\infty^\ac)$ as well). On the other hand, \emph{twisted} versions of Theorem~C and the rank part of  Theorem~A are also contained in \cite{BuyLei-SS}. Their methods share with ours that use of Beilinson--Flach classes and their explicit reciprocity laws, but they differ in a critical aspect:  
we only need to apply the method Euler/Kolyvagin systems to the plus/minus Heegner points $\mathbf{z}^\pm$ constructed in this paper,
whereas in \cite{BuyLei-SS} this method is applied to a variant of the classes $\mathcal{BF}^\pm$. 
As a consequence, when specialized to the setting considered here, the main results in \cite{BuyLei-SS} are for the twists of $E$ by  ``$p$-distinguished'' characters, whereas such twists can be avoided here. 
\sk


\noindent\emph{Acknowledgements.} A substantial part of this paper was written while the first named author visited the Morningside Center of Mathematics 
during March 2016, and he would like to thank Professor Ye Tian and the Chinese Academy of Sciences for their  hospitality and support. The second named author is partially supported by the Chinese Academy of Science grant Y729025EE1, NSFC grant 11688101, 11621061 and an NSFC grant associated to the Recruitment Program of Global Experts. Finally, we would also like to thank the anonymous referees for a very careful reading of a previous version of the paper, which greatly helped us to improve  the exposition of our results.

\section{$p$-adic $L$-functions}\label{sec:L}
Throughout this section, we let $f=\sum_{n=1}^\infty a_n(f)q^n\in S_2(\Gamma_0(N_f))$
be a newform of level $N_f$, and $K$ be an imaginary quadratic field of discriminant $-D_K<0$ prime to $N_f$.
Fix a prime $p\nmid 6N_fD_K$ and a choice of complex and $p$-adic embeddings
$\bC\overset{\imath_\infty}\hookleftarrow\overline{\bQ}\overset{\imath_p}\hookrightarrow\bC_p$; since it will suffice for our applications in this paper, we also assume that the image under $\iota_p$ of 
the number field generated by the Fourier coefficients $a_n(f)$ is contained in $\bQ_p$. 


\subsection{$p$-adic Rankin--Selberg $L$-functions}\label{sec:2varL}

Let $\Xi_K$ denote the set of algebraic Hecke characters
$\psi:K^\times\backslash\mathbb{A}_K^\times\rightarrow\bC^\times$.
We say that $\psi\in\Xi_K$ has infinity type $(\ell_1,\ell_2)\in\bZ^2$ if
\[
\psi_\infty(z)=z^{\ell_1}\overline{z}^{\ell_2},
\]
where for each place $v$ of $K$, we let $\psi_v:K_v^\times\rightarrow\bC^\times$ be
the $v$-component of $\psi$. The conductor of $\psi$ is the largest ideal $\frakc_\psi\subset\cO_K$ such that
$\psi_\frakq(u)=1$ for all $u\in(1+\frakc_\psi\cO_{K,\frakq})^\times\subset K_\frakq^\times$.
If $\psi$ has conductor $\frakc_\psi$ and
$\fraka$ is any fractional ideal of $K$ prime to $\frakc_\psi$,
we write $\psi(\fraka)$ for $\psi(a)$, where $a$ is an idele satisfying $a\hat\cO_K\cap K=\fraka$ and
such that $a_\frakq=1$ for all $\frakq\mid\frakc_\psi$.
As a function on fractional ideals, then $\psi$ satisfies
\[
\psi((\alpha))=\alpha^{-\ell_1}\overline{\alpha}^{-\ell_2}
\]
for all $\alpha\in K^\times$ with $\alpha\equiv 1\pmod{\frakc_\psi}$.

We say that a Hecke character $\psi$ of infinity type $(\ell_1,\ell_2)$ is
\emph{critical for $f$} if $s=1$ is a critical value in the sense of Deligne
for
\[
L(f/K,\psi,s)=L\left(\pi_f\times\pi_{\psi},s+\frac{\ell_1+\ell_2-1}{2}\right),
\]
where $L(\pi_f\times\pi_{\psi},s)$ is the $L$-function for the Rankin--Selberg
convolution of the cuspidal automorphic representations of ${\rm GL}_2(\mathbb{A})$
associated with $f$ and the theta series of $\psi$, respectively. 
The set of infinity types of critical characters can be written as the disjoint union
\[
\Sigma=\Sigma^{-}\sqcup\Sigma^{+}\sqcup\Sigma^{+'},
\]
with $\Sigma^{-}=\{(0,0)\}$,
$\Sigma^{+}=\{(\ell_1,\ell_2)\colon \ell_1\leqslant -1, \ell_2\geqslant 1\}$,
$\Sigma^{+'}=\{(\ell_1,\ell_2)\colon \ell_2\leqslant -1, \ell_1\geqslant 1\}$.

The involution $\psi\mapsto\psi^\rho$ on $\Xi_K$, where $\psi^\rho$ is obtained by composing
$\psi$ with the complex conjugation on $\mathbb{A}_K^\times$, has the effect on infinity types
of interchanging the regions $\Sigma^{+}$ and $\Sigma^{+'}$ (while leaving $\Sigma^{-}$ stable).
Since the values $L(f/K,\psi,1)$ and $L(f/K,\psi^\rho,1)$ are the same, for the purposes of
$p$-adic interpolation we may restrict our attention
to the first two subsets in the above decomposition of $\Sigma$. 

\begin{defn}
Let $\psi=\psi^{\infty}\psi_\infty\in\Xi_K$ be an algebraic Hecke character
of infinity type $(\ell_1,\ell_2)$. The \emph{$p$-adic avatar}
$\hat{\psi}:K^\times\backslash\hat{K}^\times\rightarrow\bC_p^\times$
of $\psi$ is defined by
\[
\hat\psi(z)=\imath_p\imath_\infty^{-1}(\psi^{\infty}(z))z_\pp^{\ell_1}z_{\overline{\pp}}^{\ell_2}.
\]
\end{defn}

For each ideal $\frakm\subset\cO_K$ let $H_\frakm={\rm Gal}(K(\frakm)/K)$ be 
the ray class group of $K$ modulo $\frakm$, and set $H_{\frakf p^\infty}=\varprojlim_rH_{\frakf p^r}$.
Via the 
Artin reciprocity map, the correspondence $\psi\mapsto\hat\psi$ establishes a
bijection between the set of algebraic Hecke characters of $K$ of conductor dividing $\frakf p^\infty$ and
the set of locally algebraic $\overline{\bQ}_p$-valued characters of $H_{\frakf p^\infty}$.

Following the terminology in \cite[\S{2.3}]{Loeffler}, for any $r, s\in\mathbf{R}_{\geqslant 0}$ we let $D^{(r,s)}(H_{\frakf p^\infty})$ be the space of $\bQ_p$-valued distributions on $H_{\frakf p^\infty}$
of order $(r,s)$ with respect to the quasi-factorization of $H_{\frakf p^\infty}$ induced by the ray class groups
$H_{\frakf\pp^\infty}$ and $H_{\frakf\overline{\pp}^\infty}$ (see [\emph{loc.cit.}, Prop.~4]).


On the other hand, let
\begin{equation}\label{eq:Hida-period}
\Omega_f^{\tt Hida}:=\frac{8\pi^2\langle f,f\rangle}{c_f}\in\overline{\bQ}_p^\times
\end{equation}
be Hida's canonical period, where 
\[
\langle f,f\rangle=\int_{\Gamma_0(N)\backslash\mathfrak{H}}\vert f(z)\vert^2dxdy 
\] 
is the Petersson norm of $f$, and $c_f$ is the congruence number 
of $f$ (\emph{cf.} \cite[\S{9.3}]{skinner-zhang}, where $c_f$ is denoted by $\eta_f(N)$). 


\begin{thm}\label{thm:hida}
Assume that $p=\frakp\overline{\pp}$ splits in $K$, let $\alpha$ and $\beta$
be the roots of $x^2-a_p(f)x+p$, and set $r:=v_p(\alpha)$ and $s:=v_p(\beta)$.
\begin{itemize}
\item[(i)]{}
There exists an element $L_p(f/K,\Sigma^+)\in{\rm Frac}(\ro[[H_{\frakf p^\infty}]]\otimes_{\bZ_p}\bQ_p)$
such that for every $\psi\in\Xi_K$ of trivial conductor 
and infinity type $(\ell_1,\ell_2)\in\Sigma^{+}$, we have
\begin{align*}
L_p(f/K,\Sigma^{+})(\hat\psi)&=
\frac{\Gamma(\ell_2)\Gamma(\ell_2+1)\cdot\mathcal{E}(f,\psi)}
{(1-\psi^{1-\rho}(\pp))(1-p^{-1}\psi^{1-\rho}(\pp))}
\cdot\frac{L(f/K,\psi,1)}{(2\pi)^{2\ell_2+1}\cdot\langle\theta_{\psi_{\ell_2}},\theta_{\psi_{\ell_2}}\rangle_{}},
\end{align*}
where $\theta_{\psi_{\ell_2}}$ is the theta series of weight $\ell_2-\ell_1+1\geqslant 3$ associated to the
Hecke character $\psi_{\ell_2}:=\psi{\rm\mathbf{N}}_K^{\ell_2}$ of infinity type $(\ell_1-\ell_2,0)$, and
\begin{equation}\label{eq:Euler-Hida}
\mathcal{E}(f,\psi)=(1-p^{-1}\psi(\pp)\alpha)(1-p^{-1}\psi(\pp)\beta)
(1-\psi^{-1}(\overline\pp)\alpha^{-1})(1-\psi^{-1}(\overline\pp)\beta^{-1}).\nonumber
\end{equation}
\item[(ii)]{}
If $r<1$ and $s<1$, then for each $\underline\alpha:=(\alpha_\pp,\alpha_{\overline{\pp}})\in
\{(\alpha,\alpha), (\alpha,\beta), (\beta,\alpha), (\beta,\beta)\}$
there exists an element $L_{p,\underline{\alpha}}(f/K,\Sigma^{-})\in D^{(r,s)}(H_{\frakf p^\infty})$
such that for every finite order character $\psi\in\Xi_K$ of conductor 
$\mathfrak{c}_\psi\mid p^\infty$, we have
\begin{align*}
L_{p,\underline{\alpha}}(f/K,\Sigma^{-})(\hat\psi)&=
\Biggl(\prod_{\frakq\mid p}\alpha_\frakq^{-v_\frakq(\mathfrak{c}_\psi)}\Biggr)
\cdot\frac{\mathcal{E}(\psi,f)}{\mathfrak{g}(\psi)\cdot\vert\mathfrak{c}_\psi\vert}
\cdot\frac{L_{}(f/K,\psi,1)}{\Omega_f^{\tt Hida}},
\end{align*}
where
\[
\mathcal{E}(\psi,f)=\prod_{\frakq\mid p,\;\frakq\nmid\mathfrak{c}_\psi}(1-\alpha_\frakq^{-1}\psi(\frakq))
(1-\alpha_{\frakq}^{-1}\psi^{-1}(\frakq)).
\]
\end{itemize}
\end{thm}

\begin{proof}
The first part is a reformulation of \cite[Thm.~6.1.3(i)]{LLZ-K}, and
the second follows from \cite[Thm.~9]{Loeffler} and [\emph{loc.~cit.}, Prop.~7]. (See also Remark~\ref{rem:periods} below.)
\end{proof}

\subsection{The two-variable plus/minus $p$-adic $L$-functions}\label{subsec:Loeffler}

Let $\Phi_n(X)=\sum_{i=0}^{p-1}X^{p^{n-1}i}$ be the $p^n$-th cyclotomic polynomial.
Fix a topological generator $\gamma_v\in H_{v^\infty}$ for each prime $v\mid p$, and define
the `half-logarithms'
\[
\log_{v}^+:=\frac{1}{p}\prod_{m=1}^\infty\frac{\Phi_{2m}(\gamma_v)}{p},
\quad\quad
\log_{v}^-:=\frac{1}{p}\prod_{m=1}^\infty\frac{\Phi_{2m-1}(\gamma_v)}{p}.
\]
These are elements in $D^{1/2}(H_{v^\infty})$ which will be seen
in $D^{(1/2,1/2)}(H_{p^\infty})$ via pullback.

\begin{thm}\label{prop:loeffler+-}
Assume that $a_p(f)=0$. Then there exist four bounded $\bQ_p$-valued distributions on $H_{p^\infty}$:
\[
L_p^{+,+}(f/K),\;L_p^{-,+}(f/K),\;L_p^{+,-}(f/K),\;L_p^{-,-}(f/K)
\]
such that for every $\underline{\alpha}=(\alpha_\pp,\alpha_{\overline{\pp}})$ as in Theorem~\ref{thm:hida}  
we have
\begin{equation}\nonumber
\begin{split}\label{dec:L}
L_{p,\underline{\alpha}}(f/K,\Sigma^{-})=&\;L_p^{+,+}(f/K)\cdot\log^+_{\pp}\log^+_{\overline\pp}\\
+&L_p^{-,+}(f/K)\cdot\log^-_{\pp}\log^+_{\overline\pp}\cdot\alpha_{\pp}
+L_p^{+,-}(f/K)\cdot\log^+_{\pp}\log^-_{\overline\pp}\cdot\alpha_{\overline{\pp}}\\
+&L_p^{-,-}(f/K)\cdot\log^-_\pp\log_{\overline{\pp}}^-\cdot\alpha_\pp\alpha_{\overline{\pp}}.
\end{split}
\end{equation}
Moreover, if 
$\phi$ is a finite order character of $H_{p^\infty}$ of conductor $\pp^{n_\pp}\overline{\pp}^{n_{\overline{\pp}}}$
with $n_\pp, n_{\overline{\pp}}>0$, then $L_p^{\bullet,\circ}(f/K)$ vanishes at $\phi$
unless $\bullet=(-1)^{n_\pp}$ and $\circ=(-1)^{n_{\overline{\pp}}}$.
\end{thm}

\begin{proof}
This is shown in \cite[\S{5}]{Loeffler}.
For our later use, we record the construction of the four
$L_{p}^{*,\circ}(f/K)$ as an explicit linear combination of the four
$L_{p,\underline{\alpha}}:=L_{p,\underline{\alpha}}(f/K,\Sigma^{-})$.
Fix a root $\alpha$ of the Hecke polynomial $x^2-a_p(f)x+p=x^2+p$, and let $\beta$ be the other root. Then:
\begin{equation}\nonumber
\begin{split}\label{def:Lp-pm}
L_p^{+,+}(f/K)&=\frac{L_{p,(\alpha,\alpha)}+L_{p,(\beta,\alpha)}
+L_{p,(\alpha,\beta)}+L_{p,(\beta,\beta)}}{4\log_{\pp}^+\log_{\overline{\pp}}^+},\\
L_p^{-,+}(f/K)&=\frac{L_{p,(\alpha,\alpha)}-L_{p,(\beta,\alpha)}
+L_{p,(\alpha,\beta)}-L_{p,(\beta,\beta)}}{4\log_{\pp}^-\log_{\overline{\pp}}^+\cdot\alpha},\\
L_p^{+,-}(f/K)&=\frac{L_{p,(\alpha,\alpha)}+L_{p,(\beta,\alpha)}
-L_{p,(\alpha,\beta)}-L_{p,(\beta,\beta)}}{4\log_{\pp}^+\log_{\overline{\pp}}^-\cdot\alpha},\\
L_p^{-,-}(f/K)&=\frac{L_{p,(\alpha,\alpha)}-L_{p,(\beta,\alpha)}
-L_{p,(\alpha,\beta)}+L_{p,(\beta,\beta)}}{4\log_{\pp}^-\log_{\overline{\pp}}^-\cdot\alpha^2}.
\end{split}
\end{equation}
Using the relation $\beta=-\alpha$, it is immediate to check
that the four identities $(\ref{dec:L})$ hold.
\end{proof}

\begin{rem}\label{rem:periods}
	In their original construction in \cite{Loeffler}, the $p$-adic $L$-functions $L_{p,\underline{\alpha}}(f/K,\Sigma^{-})$ 
	are normalized with a period $\Omega_{\Pi}$ attached to the base change to $K$ of the cuspidal automorphic representation of ${\rm GL}_2(\mathbb{A})$ associated with $f$. However, it is easy to see that Loeffler's $\Omega_{\Pi}$ agrees with our $\Omega_f^{\tt Hida}$ up to a nonzero factor in $\bQ^\times$, and so the conclusion of Theorem~\ref{prop:loeffler+-} also holds for our periods. Moreover,  
	under mild hypotheses 
	one can show that with our normalizations the elements $L_p^{\bullet,\circ}(f/K)$ are in fact integral (see Corollary~\ref{cor:kimIMC}).
\end{rem}

\subsection{Anticyclotomic $p$-adic $L$-functions}\label{sec:anti-L}

Write
\[
N_f=N^+N^-
\]
with $N^+$ (resp. $N^-$) only divisible by primes which split (resp. remain inert) in $K$. Similarly as in the Introduction, 
we say that the pair $(f,K)$ satisfies the \emph{generalized Heegner hypothesis} if
\begin{equation}\label{def:heeg}
\textrm{$N^-$ is the square-free product of an even number of primes.}\tag{{Heeg}}
\end{equation}

Let $K_\infty/K$ be the $\bZ_p^2$-extension of $K$, and set $\Gamma_K={\rm Gal}(K_\infty/K)$.
We may decompose
\[
H_{\frakf p^\infty}\simeq\Delta\times\Gamma_K
\]
with $\Delta$ a finite group.
%
The Galois group ${\rm Gal}(K/\bQ)$ acts on $\Gamma_K$ by conjugation.
Let $\Gamma^{\rm cyc}\subseteq\Gamma_K$ be the fixed part by this action, and set $\Gamma^\ac:=\Gamma_K/\Gamma^{\rm cyc}$.
Then $\Gamma^\ac\simeq{\rm Gal}(K^\ac_\infty/K)$ is the Galois group of the \emph{anticyclotomic} $\bZ_p$-extension of $K$,
on which we have $\tau\sigma\tau^{-1}=\sigma^{-1}$ for the non-trivial element $\tau\in{\rm Gal}(K/\bQ)$.
Similarly, we say that a character $\psi$ of $G_K$ is \emph{anticyclotomic} if
$\psi(\tau\sigma\tau^{-1})=\psi^{-1}(\sigma)$ for all $\sigma$, i.e., $\psi^\rho=\psi^{-1}$.

Let $L_{p,\underline{\alpha}}^\ac(f/K)$ be the image of the $p$-adic $L$-function
$L_{p,\underline{\alpha}}(f/K,\Sigma^{-})$
of Theorem~\ref{thm:hida} under the natural projection
$D^{(r,s)}(H_{\frakf p^\infty})\rightarrow D^{(r,s)}(\Gamma^\ac)$.

\begin{prop}\label{thm:hida-1}
If 
${\rm (Heeg)}$ holds, then $L_{p,(\alpha,\alpha)}^\ac(f/K)$ and $L_{p,(\beta,\beta)}^\ac(f/K)$ 
are identically zero.
\end{prop}



\begin{proof}
Via Rankin--Selberg convolution techniques, B.-D. Kim has constructed in
\cite{kim-JNT} $p$-adic $L$-functions $\mathcal{L}_{p,(\alpha,\alpha)}(f/K)$
and $\mathcal{L}_{p,(\beta,\beta)}(f/K)$ which are easily seen to be nonzero constant multiples
of Loeffler's $L_{p,(\alpha,\alpha)}(f/K,\Sigma^{-})$ and $L_{p,(\beta,\beta)}(f/K,\Sigma^{-})$, respectively
(see the remarks in \cite[p.~378]{Loeffler}). Via the usual identifications $\bQ_p[[H_{\pp^\infty}]]\simeq \bQ_p[[X]]$
and $\bQ_p[[H_{\overline{\pp}^\infty}]]\simeq\bQ_p[[Y]]$ sending $\gamma_\pp\mapsto 1+X$ and $\gamma_{\overline{\pp}}\mapsto 1+Y$,
we may view these $p$-adic $L$-functions as two-variable power series in the variables $X$ and $Y$.
Let $\varepsilon_K$ denote the quadratic character associated with $K$ by class field theory.
The same argument as in \cite[Thm.~1.1]{PR89} and \cite[\S{4.2}]{disegni}, then shows
that $\mathcal{L}_{p,(\alpha,\alpha)}(f/K)$ (and hence also $L_{p,(\alpha,\alpha)}(f/K,\Sigma^{-}))$
satisfies the functional equation
\begin{equation}\label{eq:fun-eq}
\mathcal{L}_{p,(\alpha,\alpha)}(f/K)\left(\frac{1}{1+Y}-1,\frac{1}{1+X}-1\right)=\epsilon\mathcal{L}_{p,(\alpha,\alpha)}(f/K)(X,Y),
\end{equation}
where $\epsilon=-\varepsilon_K(N)$, and similarly for $\mathcal{L}_{p,(\beta,\beta)}(f/K)$.
Since the change of variables $(X,Y)\mapsto(\frac{1}{1+Y}-1,\frac{1}{1+X}-1)$ corresponds
to the transformation $\phi\mapsto\phi^{-\rho}$ on characters of $H_{p^\infty}$,
this shows that under the generalized Heegner hypothesis 
(which implies $\epsilon=-1$) both $L_{p,(\alpha,\alpha)}(f/K,\Sigma^{-})$ and $L_{p,(\beta,\beta)}(f/K,\Sigma^{-})$ vanish at all
anticyclotomic characters of $H_{p^\infty}$, whence the result.
\end{proof}

Throughout the following, 
we shall identify the space of bounded
$\bQ_p$-valued measures on a compact $p$-adic Lie group $G$ with the Iwasawa algebra
$\bQ_p\otimes_{\bZ_p}\bZ_p[[G]]$. Viewing the measures $L_p^{\bullet,\circ}(f/K)$ of Theorem~\ref{prop:loeffler+-}
as elements in $\bQ_p\otimes_{\ro}\ro[[H_{\frakf p^\infty}]]$, we thus denote by $L_{p,{\rm ac}}^{\bullet,\circ}(f/K)$
their images under the natural projection
$\bQ_p\otimes_{\ro}\ro[[H_{\frakf p^\infty}]]\rightarrow\bQ_p\otimes_{\bZ_p}\ro[[\Gamma^\ac]]$.

\begin{cor}\label{cor:signed-0}
Assume that 
$a_p(f)=0$.
If 
${\rm (Heeg)}$ holds, then $L_{p,{\rm ac}}^{\eps,\eps}(f/K)$ is identically zero for all $\eps\in\{+,-\}$.
\end{cor}

\begin{proof}
As in \cite[Thm.~5.13]{pollack}, the idea is to use
the decomposition in Proposition~\ref{prop:loeffler+-} to deduce
from the functional equation for $L_{p,\underline{\alpha}}(f/K,\Sigma^{-})$
a similar one for $L_p^{\varepsilon,\varepsilon}(f/K)$ forcing the vanishing of
$L_{p,{\rm ac}}^{\varepsilon,\varepsilon}(f/K)$ under our generalized Heegner hypothesis.

Indeed, writing the functional equation $(\ref{eq:fun-eq})$ for $L_{p,(\alpha,\alpha)}(f/K,\Sigma^{-})$
in terms of the signed $p$-adic $L$-functions $L_p^{\bullet,\circ}:=L_p^{\bullet,\circ}(f/K)$ we obtain
\begin{equation}
\begin{split}\label{eq:fun-eq-pm}
&\log_{\pp}^+\log_{\overline{\pp}}^+\cdot\left(L_p^{+,+}(X,Y)-\epsilon L_p^{+,+}\left(\frac{1}{1+Y}-1,\frac{1}{1+X}-1\right)\right)\\
&+\log_{\pp}^-\log_{\overline{\pp}}^-\left(L_p^{-,-}(X,Y)-\epsilon L_p^{-,-}\left(\frac{1}{1+Y}-1,\frac{1}{1+X}-1\right)\right)\cdot\alpha^2\\
=&\log_{\pp}^+\log_{\overline{\pp}}^-\left(-L_p^{+,-}(X,Y)+\epsilon L_p^{+,-}\left(\frac{1}{1+Y}-1,\frac{1}{1+X}-1\right)\right)\cdot\alpha\\
&+\log_{\pp}^-\log_{\overline{\pp}}^+\left(-L_p^{-,+}(X,Y)+\epsilon L_p^{-,+}\left(\frac{1}{1+Y}-1,\frac{1}{1+X}-1\right)\right)\cdot\alpha,
\end{split}
\end{equation}
where $\epsilon=-\varepsilon_K(N)$. Since ${\rm ord}_p(\alpha)=1/2$, the nonzero coefficients in the
left-hand side of this equality have coefficients with $p$-adic valuations in $\bZ$, whereas
the nonzero coefficients in the right-hand side have $p$-adic valuations in $\frac{1}{2}\bZ\smallsetminus\bZ$.
This forces both sides to be identically zero, and so we obtain
\[
L_p^{+,+}(f/K)\left(\frac{1}{1+Y}-1,\frac{1}{1+X}-1\right)=\epsilon L_p^{+,+}(f/K)(X,Y),
\]
and similarly for $L_p^{-,-}(f/K)$. Since $({\rm Heeg})$ implies that $\epsilon=-1$, this shows that $L_{p,\ac}^{+,+}(f/K)$ and $L_{p,\ac}^{-,-}(f/K)$ are then identically zero, as was to be shown. 
\end{proof}

The following anticyclotomic $p$-adic $L$-function will play a key 
role in this paper. 
Let $\unr$ denote the completion of the ring of integers of the maximal unramified extension of $\bQ_p$.

\begin{thm}\label{thm:bdp}
Assume hypothesis $({\rm Heeg})$ and that $p=\pp\overline{\pp}$ splits in $K$. Then there exists a $p$-adic $L$-function
$\mathscr{L}^{\tt BDP}_{\pp}(f/K)\in\unr[[\Gamma^\ac]]$ such that if
$\hat{\psi}:\Gamma^\ac\rightarrow\bC_p^\times$ has trivial conductor and
infinity type $(-\ell,\ell)$ with $\ell\geqslant 1$, then
\begin{align*}
\biggl(\frac{\mathscr{L}^{\tt BDP}_{\pp}(f/K)(\hat\psi)}{\Omega_p^{2\ell}}\biggr)^2&=
\Gamma(\ell)\Gamma(\ell+1)
\cdot(1-p^{-1}\psi(\pp)\alpha)^2(1-p^{-1}\psi(\pp)\beta)^2
\cdot\frac{L(f/K,\psi,1)}{\pi^{2\ell+1}\cdot\Omega_K^{4\ell}},
\end{align*}
where 
$\Omega_p\in\unr^\times$ and $\Omega_K\in\bC^\times$ are CM periods attached to $K$.
\end{thm}

\begin{proof}
This follows from the results in \cite[\S{3.3}]{cas-hsieh1}. (See the proof of Theorem~\ref{3.1} below
for the precise relation between the construction in \emph{loc.~cit.} and the above $\mathscr{L}_\pp^{\tt BDP}(f/K)$.)
\end{proof}

We next recall some of the nontriviality properties one knows about $\mathscr{L}_{\pp}^{\tt BDP}(f/K)$, which we will also need. 
Let $\rho_f:{\rm Gal}(\overline{\bQ}/\bQ)\rightarrow{\rm Aut}_{\bQ_p}(V_f)\simeq{\rm GL}_2(\bQ_p)$
be the $p$-adic Galois representation attached to $f$, and let $\overline{\rho}_f$ denote
its associated semi-simple mod $p$ representation.

\begin{thm}\label{thm:mu-bdp}
In addition to the hypotheses in Theorem~\ref{thm:bdp}, assume that:
\begin{itemize}
\item{} $\overline{\rho}_f\vert_{{\rm Gal}(\overline{\bQ}/K)}$ is absolutely irreducible,
\item{} $\overline{\rho}_f$ is ramified at every prime $\ell\mid N^-$.
\end{itemize}
Then $\mathscr{L}^{\tt BDP}_{\pp}(f/K)$ is not identically zero, and it has trivial $\mu$-invariant.
\end{thm}

\begin{proof}
The nonvanishing of $\mathscr{L}_\pp^{\tt BDP}(f/K)$ follows from \cite[Thm.~3.7]{cas-hsieh1},
where it is deduced from \cite[Thm.~C]{hsieh}. The vanishing of $\mu(\mathscr{L}^{\tt BDP}_{\pp}(f/K))$
similarly follows from \cite[Thm.~B]{hsieh} (or alternatively, from \cite[Thm.~B]{burungale-II}
in the cases where the number of prime factors in $N^-$ is positive, noting that by
the discussion in \cite[p.~912]{prasanna} our last assumption guarantees that the term $\alpha(f,f_B)$
in \cite[Thm.~5.6]{burungale-II} is a $p$-adic unit).
\end{proof}

Letting $L_{p,\ac}(f/K)$ be the image of the $p$-adic $L$-function
$L_p(f/K,\Sigma^{+})$ of Theorem~\ref{thm:hida} under the natural
map induced by the natural projection $H_{\frakf p^\infty}\rightarrow\Gamma^{\rm ac}$, 
we see that $L_{p,\ac}(f/K)$ and \emph{the square} of the $p$-adic $L$-function $\mathscr{L}^{\tt BDP}_{\pp}(f/K)$ 
are defined by the interpolation of the same $L$-values.
However, the archimedean periods used in their construction are different, and so
these $p$-adic $L$-functions need not be equal (even up to units in the Iwasawa algebra).
In fact, as shown in Theorem~\ref{thm:factorization} below,
the ratio between these different periods is interpolated by an anticyclotomic projection of a
Katz $p$-adic $L$-function.\footnote{This phenomenon appears to have been first observed by Hida--Tilouine
\cite[\S{8}]{HT-ENS} in a slightly different context; see also \cite[\S{3.2}]{DLR}.}

Before we state the defining property of the Katz $p$-adic $L$-function, recall that
the Hecke $L$-function of $\psi\in\Xi_K$
is defined by (the analytic continuation of) the Euler product
\[
L_{}(\psi,s)=\prod_{\frakl}\biggl(1-\frac{\psi(\frakl)}{N(\frakl)^s}\biggr)^{-1},
\]
where 
$\frakl$ runs over all prime ideals of $K$, with the convention
that $\psi(\frakl)=0$ for $\frakl\mid\mathfrak{c}_\psi$. The set of infinity types
of $\psi\in\Xi_K$ for which $s=0$ is a critical value of $L(\psi,s)$ 
can be written as the disjoint union $\Sigma_K\sqcup\Sigma_K'$, where $\Sigma_K=\{(\ell_1,\ell_2):\ell_2\leqslant 0<\ell_1\}$ and 
$\Sigma_K'=\{(\ell_1,\ell_2):\ell_1\leqslant 0<\ell_2\}$.

\begin{thm}\label{thm:katz}
Assume that $p=\pp\overline{\pp}$ splits in $K$. Then
there is a $p$-adic $L$-function
$\mathscr{L}^{\tt Katz}_{\pp}(K)\in\unr[[H_{\frakf p^\infty}]]$
such that if $\psi\in\Xi_K$ has trivial conductor 
and infinity type $(\ell_1,\ell_2)\in\Sigma_K$, then
\[
\frac{\mathscr{L}^{\tt Katz}_{\pp}(K)(\hat\psi)}{\Omega_p^{\ell_1-\ell_2}}=\biggl(\frac{\sqrt{D_K}}{2\pi}\biggr)^{\ell_2}
\cdot\Gamma(\ell_1)\cdot(1-\psi(\overline{\pp}))(1-p^{-1}\psi^{-1}(\pp))
\cdot\frac{L_{}(\psi,0)}{\Omega_K^{\ell_1-\ell_2}},
\]
where $\Omega_p$ and $\Omega_K$ are as in Theorem~\ref{thm:bdp}.
\end{thm}

\begin{proof}
See \cite[\S{5.3.0}]{Katz49}, or \cite[Thm.~II.4.14]{de_shalit}.
\end{proof}

Denote by $\mathscr{L}_{\pp,\ac}^{\tt Katz}(K)$
the image of $\mathscr{L}^{\tt Katz}_\pp(K)$ under the projection
$\unr[[H_{\frakf p^\infty}]]\rightarrow\unr[[\Gamma^\ac]]$.

\begin{thm}\label{thm:factorization}
Assume hypothesis $({\rm Heeg})$ and that $p=\pp\overline{\pp}$ splits in $K$.  
Then
\[
L_{p,\ac}(f/K)(\hat\psi)
=\frac{w_K}{h_K}\cdot\frac{\mathscr{L}^{\tt BDP}_{\pp}(f/K)^2(\hat\psi)}{\mathscr{L}_{\pp,\ac}^{\tt Katz}(K)(\hat\psi^{\rho-1})}
\]
up to a unit in $\ro[[\Gamma^\ac]]^\times$, where $w_K=\vert\cO_K^\times\vert$ and
$h_K$ is the class number of $K$.
\end{thm}

\begin{proof}
This is \cite[Thm.~1.7]{cas-BF}, whose proof does not make use of the underlying 
ordinarity hypothesis on $f$ made in \emph{loc.cit.}. See also \cite[\S{5.3}]{JSW} for a similar calculation.
\end{proof}

\subsection{Another $p$-adic Rankin--Selberg $L$-function}\label{subsec:Wan}

Recall the decomposition $H_{\frakf p^\infty}\simeq\Delta\times\Gamma_K$, set
$\Lambda:=\ro[[\Gamma_K]]$,  $\Lambda_{\unr}:=\unr[[\Gamma_K]]$, and $\Lambda_{\rm ac}:=\ro[[\Gamma^{\rm ac}]]$,  and 
continue to denote by
\[
L_{p}(f/K,\Sigma^{+})\in{\rm Frac}(\Lambda\otimes_{\bZ_p}\bQ_p)\quad\quad\textrm{and}\quad\quad
\mathscr{L}_{\pp}^{\tt Katz}(K)\in\Lambda_{\unr}
\]
the natural projections of the $p$-adic $L$-functions
$L_p(f/K,\Sigma^{+})$ and $\mathscr{L}^{\tt Katz}_\pp(K)$ of
Theorem~\ref{thm:hida} and Theorem~\ref{thm:katz}, respectively.

\begin{thm}\label{thm:wan}
Assume that $p=\pp\overline{\pp}$ splits in $K$. There exists a $p$-adic $L$-function
\[
L_{\pp}(f/K)\in\Lambda_{\unr}
\]
such that if $\hat{\psi}:\Gamma\rightarrow\bC_p^\times$ has trivial conductor and
infinity type $(\ell_1,\ell_2)\in\Sigma^{+}$, then
\begin{align*}
L_{\pp}(f/K)(\hat\psi)=
&\frac{\Gamma(\ell_2)\Gamma(\ell_2+1)}{\pi^{2\ell_2+1}}
\cdot\mathcal{E}(f,\psi)
\cdot\frac{\Omega_p^{2(\ell_2-\ell_1)}}{\Omega_K^{2(\ell_2-\ell_1)}}\cdot L(f/K,\psi,1),
\end{align*}
where $\mathcal{E}(f,\psi)=(1-p^{-1}\psi(\pp)\alpha)(1-p^{-1}\psi(\pp)\beta)
(1-\psi^{-1}(\overline\pp)\alpha^{-1})(1-\psi^{-1}(\overline\pp)\beta^{-1})$,
and $\Omega_K$ and $\Omega_p$ are as in Theorem~\ref{thm:bdp}. Moreover, 
$L_\pp(f/K)$ differs from the product
\begin{equation}\label{eq:product}
\widetilde{L}_{\pp}(f/K)(\hat\psi):=
L_{p}(f/K,\Sigma^{+})(\hat\psi)
\cdot\frac{h_K}{w_K}\cdot\mathscr{L}_{\pp,\ac}^{\tt Katz}(K)(\hat\psi^{\rho-1})
\end{equation}
by a unit in $\Lambda_{}^\times$, and it is not identically zero.
\end{thm}

\begin{proof}
The construction of $L_\pp(f/K)$ is given in \cite[\S{4.6}]{wan-combined}.
On the other hand, the fact that the product $(\ref{eq:product})$ has the claimed interpolation
property follows by a straightforward adaptation of the calculations in \cite[Thm.~1.7]{cas-BF} (or \cite[\S{5.3}]{JSW}).
Finally, the fact that $L_\pp(f/K)$ is not the zero function follows from the fact that for some
of the characters $\psi$ in the range of $p$-adic interpolation, the Euler product defining
$L(f/K,\psi,s)$ converges at $s=1$.
\end{proof}

\begin{cor}\label{cor:wan-bdp}
Assume hypothesis $({\rm Heeg})$ and that $p=\pp\overline{\pp}$ splits in $K$, and
let $L_{\pp,\ac}(f/K)$ be the image of the $p$-adic
$L$-function $L_\pp(f/K)$ of Theorem~\ref{thm:wan} under the 
projection
$\Lambda_{\unr}\rightarrow\Lambda_{\unr}^\ac$. Then
\[
L_{\pp,\ac}(f/K)=\mathscr{L}^{\tt BDP}_\pp(f/K)^2
\]
up to a unit in $\Lambda_{\ac}^\times$, where $\mathscr{L}_\pp^{\tt BDP}(f/K)$ is as in Theorem~\ref{thm:bdp}.
\end{cor}

\begin{proof}
This follows from a direct comparison of their interpolation properties.
\end{proof} 

\section{Selmer groups}\label{sec:Sel}
\def\K{K}
\def\ac{{\rm ac}}
\def\ur{{\rm ur}}
\def\T{T}
\def\V{V}
\def\A{A}

\subsection{Local conditions at $p$}\label{subsec:local-p}

In this section, we develop some local results for studying the anticyclotomic Iwasawa theory for elliptic curves at
supersingular primes. Throughout, we let $E/\bQ$ be an elliptic curve of conductor $N$, $p>3$ be a prime of good supersingular reduction for $E$, and $K/\bQ$ be an imaginary quadratic field of discriminant
prime to $N$ and such that
\[
p=\pp\overline{\pp}\quad\textrm{splits in $\K$}.
\]
We keep the notations introduced in Section~\ref{sec:L}; in particular, we have $a_p(f)=0$, and $K_\infty^\ac/K$ denotes the anticyclotomic $\bZ_p$-extension of $K$.

It is easy to see that every prime $v\mid p$ is finitely decomposed in $\K_\infty^{\ac}/\K$,
say as the product $v_1 v_2\cdots v_{p^t}$; then $v$ is also decomposed into $p^t$ distinct primes in $\K_\infty/\K$.
Let $\Gamma_1$ (resp. $\Gamma_1^{\ac}$) be the decomposition group of $v_1$ in
$\Gamma_{}$ (resp. $\Gamma^{\ac}$). Let $H_\K$ be the Hilbert class field of $\K$, set
$\K_0^{\ac}:=\K_\infty^{\ac}\cap H_\K$, and let $\K_m^{\ac}$ be the unique subfield of $\K^{\ac}_\infty$ with
$[\K_m^\ac:\K^\ac_0]=p^m$. Let $a$ be the inertial degree of $\K^\ac_0/\K$ at any $v\mid p$.

We denote by $\bQ_p^\ur$ and $\bQ_{p,\infty}$ the unramified and the cyclotomic $\bZ_p$-extensions of $\bQ_p$, respectively,
and let $\bQ_{p,\infty}^\ur$ denote their composition. For $v\mid p$, we identify
$\K_v\simeq\bQ_p$. Let $u_v$ and $\gamma_v$ be topological generators of
$U_v:=\mathrm{Gal}(\bQ_{p,\infty}^\ur/\bQ_{p,\infty})$ and $\Gamma_v:=\mathrm{Gal}(\bQ_{p,\infty}^{\ur}/\bQ_p^{\ur})$;
these are chosen so that $u_v$ is the arithmetic Frobenius and (using additive notation) 
$-p^au_v+\gamma_v$ is a topological generator of
${\rm Gal}(\K_{\infty,v}/\K_{\infty,v}^\ac)$. Let $X_v=\gamma_v-1$ and $Y_v=u_v-1$.
Finally, let $\gamma_\ac\in\Gamma^{\ac}$ be a topological generator, so that
$\bZ_p[[\Gamma^\ac]]\simeq\bZ_p[[T]]$ via $\gamma_\ac\mapsto T+1$.

Following \cite{kobayashi-152} (see also \cite[\S{2.1}]{kim-CJM}),
for any unramified extension $k$ of $\bQ_p$ we define the subgroups $E^\pm(k(\mu_{p^{n+1}}))$
of $E(k(\mu_{p^{n+1}}))$ by
\begin{align*}
E^+({k(\mu_{p^{n+1}})})&:=\biggl\{ P\in E({k(\mu_{p^{n+1}})})\;\bigl\vert\;
\mathrm{tr}^{k(\mu_{p^{n+1}})}_{k(\mu_{p^{\ell+2}})}(P)\in E({k(\mu_{p^{\ell+1}})})\;
\textrm{for}\;0\leqslant \ell<n,\;\textrm{even $\ell$}\biggr\},\\
E^-({k(\mu_{p^{n+1}})})&:=\biggl\{ P\in E({k(\mu_{p^{n+1}})})\;\bigl\vert\;
\mathrm{tr}^{k(\mu_{p^{n+1}})}_{k(\mu_{p^{\ell+2}})}(P)\in E({k(\mu_{p^{\ell+1}})})\;
\textrm{for}\;-1\leqslant \ell<n,\;\textrm{odd $\ell$}\biggr\}.
\end{align*}

Letting $\wh{E}$ denote the formal group associated to the minimal model of $E$ over $\bZ_p$,
we may similarly define the subgroups $\wh{E}^\pm(\mathfrak{m}_{k(\mu_{p^{n+1}})})$
of $\wh{E}(\mathfrak{m}_{k(\mu_{p^{n+1}})})$, and 
using that $a_p:=a_p(f)=0$ one easily checks that 
\[
E^\pm({k(\mu_{p^{n+1}})})\otimes\bQ_p/\bZ_p=\wh{E}^\pm(\mathfrak{m}_{k(\mu_{p^{n+1}})})\otimes
\bQ_p/\bZ_p.
\]

Fix a compatible system $\{\zeta_{p^n}\}_{n\geqslant 0}$ of primitive $p^n$-th roots of unity
$\zeta_{p^n}$ (i.e., $\zeta_{p^{n+1}}^p=\zeta_{p^{n}}$ for $n$ and $\zeta_p\neq 1$).
Let $\varphi$ be the Frobenius on $k/\bQ_p$, and for any polynomial $f\in k[X]$ set
\[
\log_{f}(X)=\sum_{n=0}^\infty(-1)^n\frac{f^{(2n)}(X)}{p^n},
\]
where $f^{(2n)}(X)=f^{\varphi^{2n-1}}\circ\cdots\circ f^{\varphi}\circ f(X)$.
As in \cite[\S{3.2}]{kim-parity}, for any unit $z\in\mathcal{O}_k^\times$ one can construct a point
$\tilde{c}_{n,z}\in\wh{E}(\mathfrak{m}_{k(\mu_{p^n})})$ such that
\begin{equation}\label{eq:def-c}
\log_{\wh{E}}(\tilde{c}_{n,z})=\left[\sum_{i=1}^\infty(-1)^{i-1}z^{\varphi^{-(n+2i)}}\cdot p^i\right]
+\log_{f_z^{\varphi^{-n}}}(z^{\varphi^{-n}}\cdot(\zeta_{p^n}-1)),
\end{equation}
with $f_z(X):=(X+z)^p-z^p$. 

\begin{rem}
Since $\wh{E}(\mathfrak{m}_{k(\mu_{p^n})})$ is torsion-free
(see \cite[Prop.~8.7]{kobayashi-152} and \cite[Prop.~3.1]{kim-parity}),
the formal group logarithm $\log_{\wh{E}}$ is injective, and hence the point $\tilde{c}_{n,z}$ is
uniquely defined by $(\ref{eq:def-c})$.
\end{rem}

Let $k_n\subset k(\mu_{p^{n+1}})$ be the unique subfield of degree $p^n$ over $k$, and let $\mathfrak{m}_{k,n}$ be the maximal ideal of the valuation ring of $k_n$. Let $\wh{E}^{\pm}(\mathfrak{m}_{k,n})$ be the image of $\wh{E}^{\pm}(\mathfrak{m}_{k(\mu_{p^{n+1}})})$
under 
${\rm tr}^{k(\mu_{p^{n+1}})}_{k_n}$, define $E^{\pm}(k_{n})$ similarly, and set
\begin{equation}\label{eq:def-czeta}
c_{n,z}:={\rm tr}^{k(\mu_{p^{n+1}})}_{k_n}(\tilde{c}_{n+1,z})\in\wh{E}(\mathfrak{m}_{k,n}).
\end{equation}

Let $\Phi_m(X)=\sum_{i=0}^{p-1}X^{p^{m-1}i}$ be the $p^m$-th cyclotomic polynomial, define
\[
\tilde{\omega}^+_n(X):=\prod_{\substack{1\leqslant m\leqslant n\\ m\;{\rm even}}}\Phi_m(1+X),
\quad\quad\tilde{\omega}^-_n(X):=\prod_{\substack{1\leqslant m\leqslant n\\ m\;{\rm odd}}}\Phi_m(1+X),
\]
and set $\omega_n^{\pm}(X)=X\tilde{\omega}_n^\pm(X)$. We denote by $k^m$ the unramified extension of $\bQ_p$ of degree $p^m$, write
$k_{n,m}$ and $\mathfrak{m}_{n,m}$ for the above $k_n$ and $\mathfrak{m}_{k,n}$
with $k=k^m$, and set
\begin{align*}
\Lambda_{n,m}:=\bZ_p[{\rm Gal}(k_{n,m}/\bQ_p)],\quad
\Lambda_{n,m}^\pm:=&\;\bZ_p[[\Gamma_1]]/(\omega_n^\pm(X),(1+Y)^{p^m}-1)\simeq\tilde{\omega}_n^{\mp}(X)\Lambda_{n,m},
\end{align*}
where the last isomorphism follows from the relation $(1+X)^{p^n}-1=X\tilde{\omega}^+_n(X)\tilde{\omega}^-_n(X)$.





\begin{lem}\label{lem:compat}
There is a sequence of points $c_{n,m}\in\wh{E}(\mathfrak{m}_{n,m})$ satisfying
the compatibilities:
\begin{align*}
{\rm tr}^{k_{n,m}}_{k_{n,m-1}}(c_{n,m})&=c_{n,m-1},\quad\quad
{\rm tr}^{k_{n,m}}_{k_{n-1,m}}(c_{n,m})=-c_{n-2,m}.
\end{align*}
Moreover, for even (resp. odd) values of $n$, $c_{n,m}$ generates
$\wh{E}^+(\mathfrak{m}_{n,m})$ (resp. $\wh{E}^-(\mathfrak{m}_{n,m})$) as a 
$\Lambda_{n,m}$-module.
\end{lem}

\begin{proof}
The first part is shown in \cite[Lem.~6.2]{wan-combined}.
For our later use, we recall the construction of $c_{n,m}$.
By the normal basis theorem, we may fix an element
$d=\{d_m\}_m\in\varprojlim_{m}\mathcal{O}_{k^m}^\times$ 
generating $\varprojlim_{m}\mathcal{O}_{k^m}^\times$
as a $\bZ_p[[U]]$-module. Writing $d_m=\sum_j a_{m,j}\zeta_j$,
with $\zeta_j$ roots of unity and $a_{m,j}\in\bZ_p$, one then defines
\begin{equation}\label{eq:def-cmn}
c_{n,m}:=\sum_j a_{m,j}c_{n,\zeta_j},
\end{equation}
where $c_{n,\zeta_j}$ is as in $(\ref{eq:def-czeta})$. The proof of the
above trace relations then follows from an explicit calculation of the
images under $\log_{\hat{E}}$ of both sides using $(\ref{eq:def-c})$.
The second claim in the lemma is contained in  \cite[Lem.~6.4]{wan-combined}.
\end{proof}

\begin{defn}\label{def:H-pm}
Let $T$ be the $p$-adic Tate module of $E$. 
We define $H^1_{\pm}(k_{n,m},T)\subseteq H^1(k_{n,m},T)$ to be the orthogonal complement of
$E^\pm(k_{n,m})\otimes\bQ_p/\bZ_p$ under the local Tate pairing
\[
(\;,\;)_{n,m}:H^1(k_{n,m},T)\times H^1(k_{n,m},E[p^\infty])\longrightarrow\bQ_p/\bZ_p,
\]
where we view $E^\pm(k_{n,m})\otimes\bQ_p/\bZ_p$ as embedded in $H^1(k_{n,m},E[p^\infty])$
by the Kummer map.
\end{defn}

\subsection{The plus/minus Coleman maps}\label{subsec:Col-pm}

We recall Kobayashi's construction of the plus/minus
Coleman maps for the cyclotomic $\bZ_p$-extension of $\bQ_p$,
as extended by B.-D.~Kim \cite{kim-CJM} to finite unramified extensions of $\bQ_p$. 

Define the maps $P_{c_{n,m}}:H^1(k_{n,m},T)\rightarrow\Lambda_{n,m}=\bZ_p[{\rm Gal}(k_{n,m}/\bQ_p)]$ by
\[
P_{c_{n,m}}(z)=\sum_{\sigma\in{\rm Gal}(k_{n,m}/\bQ_p)}(c_{n,m}^\sigma,z)_{n,m},
\]
and set $P_{c_{n,m}}^\pm:=(-1)^{[\frac{n+1}{2}]}P_{c_{n,m}^\pm}$, where
\[
c_{n,m}^+=\left\{\begin{array}{ll}c_{n,m}&\textrm{if $n$ is even,}\\
c_{n-1,m}&\textrm{if $n$ is odd,}\end{array}\right.
\quad\quad
c_{n,m}^-=\left\{\begin{array}{ll}c_{n-1,m}&\textrm{if $n$ is even,}\\
c_{n,m}&\textrm{if $n$ is odd.}\end{array}\right.
\]

By Lemma~\ref{lem:compat}, the maps $P_{c_{n,m}}^\pm$ factor through the quotient
by $H^1_\pm(k_{n,m},T)$ and they satisfy natural compatibilities for varying $n$ and $m$.
Moreover, as shown in \cite[Thms.~2.7-8]{kim-CJM} (see also \cite[\S{8.5}]{kobayashi-152}),
there are unique maps ${\rm Col}_{n,m}^\pm$ making the following diagram commute:
\[
\xymatrix{
H^1(k_{n,m},T)\ar[rr]^-{{\rm Col}_{n,m}^\pm}\ar[d]&&\Lambda_{n,m}^\pm\ar[d]^-{\cdot\tilde{\omega}_{n}^\mp}\\
H^1(k_{n,m},T)/H^1_{\pm}(k_{n,m},T)\ar[rr]^-{P_{c_{n,m}}^\pm}&&\Lambda_{n,m}.
}
\]
The maps ${\rm Col}_{n,m}^\pm$ are isomorphisms and passing to the limit they define $\Lambda$-linear
isomorphisms
\begin{equation}\label{def:Col-pm}
{\rm Col}^\pm:\varprojlim_{n,m}\frac{H^1(k_{n,m},T)}{H^1_{\pm}(k_{n,m},T)}
\xrightarrow{\;\sim\;}\varprojlim_{n,m}\Lambda_{n,m}\simeq\bZ_p[[\Gamma_1]].
\end{equation}

\subsection{The plus/minus logarithm maps}\label{subsec:PR-maps}

We now define local big logarithm maps ${\rm Log}_{\ac}^\pm$ on
$H^1_\pm(K_v,\mathbf{T}^\ac)$, where
\begin{equation}\label{def:Tac}
\mathbf{T}^\ac:=T\otimes\bZ_p[[\Gamma^\ac]](\Psi^{-1})
\end{equation}
for the canonical character $\Psi:\Gamma^\ac\hookrightarrow\bZ_p[[\Gamma^\ac]]^\times$.
As it will be clear to the reader, these maps are the restriction to the ``anticyclotomic line''
of the two-variable plus/minus logarithm maps ${\rm Log}^\pm$ introduced in \cite[\S{6.1}]{wan-combined}.
We still keep the notations from Section~\ref{subsec:local-p}.

Via the natural inclusion
\[
E(k_{n,m})\otimes\bQ_p/\bZ_p=(E(k_{n,m})\otimes\bQ_p/\bZ_p)^\perp
\subseteq(E^\pm(k_{n,m})\otimes\bQ_p/\bZ_p)^\perp=H^1_{\pm}(k_{n,m},T),
\]
the points $c_{n,m}$ 
lie in the $\Lambda_{n,m}$-module $H_{\pm}^1(k_{n,m},T)$. By \cite[Lem.~6.9]{wan-combined}, one can choose norm-compatible classes
$b_{n,m}^\pm\in H^1_{\pm}(k_{n,m},T)$ with the property that
\[
\tilde{\omega}_n^{-\epsilon}(X)b_{n,m}^\epsilon=
(-1)^{[\frac{n+1}{2}]}c_{n,m},
\]
where $\epsilon=(-1)^n$, and such that $\varprojlim_{n,m}b_{n,m}^\pm$
generates $\varprojlim_{m,n}H^1_{\pm}(k_{n,m},T)$ as a free $\bZ_p[[\Gamma_1]]$-module of rank one.
Noting that $K^{\ac}_{m,v}\subseteq k_{m,m+a}$,
we define
\begin{align*}
E^+(K_{m,v}^\ac)&:=\left\{P\in E(K_{m,v}^\ac)\;\bigl\vert\;
\mathrm{tr}^{K_{m,v}^\ac}_{K_{\ell+1,v}^\ac}(P)\in E(K_{\ell,v}^\ac)\;
\textrm{for}\;0\leqslant \ell<m,\;\textrm{even $\ell$}\right\},\\
E^-(K_{m,v}^\ac)&:=\left\{P\in E(K_{m,v}^\ac)\;\bigl\vert\;
\mathrm{tr}^{K_{m,v}^\ac}_{K_{\ell+1,v}^\ac}(P)\in E(K_{\ell,v}^\ac)\;
\textrm{for}\;-1\leqslant \ell<m,\;\textrm{odd $\ell$}\right\},
\end{align*}
and we easily see that
\[
E^\pm(K^\ac_{m,v})\otimes\bQ_p/\bZ_p=(E^\pm(k_{m,m+a})\otimes\bQ_p/\bZ_p)\cap H^1(K_{m,v}^\ac,E[p^\infty]).
\]

Let $H^1_{\pm}(K_{m,v}^\ac,T)$ be the image of $H^1_\pm(k_{m,m+a},T)$ 
under corestriction from $k_{m,m+a}$ to $K_m^\ac$.
Set $\mathbf{T}_1^\ac=T\otimes\bZ_p[[\Gamma_1]](\can^{-1})$, where
$\Psi:\Gamma_1^\ac\hookrightarrow\bZ_p[[\Gamma_1^\ac]]^\times$ is the canonical character,
and we let $G_{K_v}$ act diagonally on the tensor product $\mathbf{T}_1^\ac$. Then
$H^1_\pm(K_v,\mathbf{T}_1^\ac)\simeq\varprojlim_m H^1_{\pm}(K_{m,v}^\ac,T)$ by Shapiro's lemma, and
the elements
\begin{equation}\label{eq:def-a}
a_m^\pm:=\mathrm{tr}^{k_{m,m+a}}_{K_{m,v}^\ac}(b^\pm_{m,m+a})\nonumber
\end{equation}
are norm-compatible, with $a^\pm:=\varprojlim a_m^\pm$
generating $H^1_\pm(K_v,\mathbf{T}_1^\ac)$
as a free $\bZ_p[[\Gamma_1^\ac]]$-module.

Recall that $v_1, v_2,\dots,v_{p^t}$ are the primes over a place $v\mid p$ in $K_\infty/K$. 
Since every prime above $p$ is totally ramified in $K_\infty/K_\infty^\ac$, by abuse of notation 
we will still denote by $v_1, v_2, \dots,v_{p^t}$ the primes above $v$ in $K_\infty^\ac/K$.
Let $\gamma_1={\rm id}, \gamma_2, \dots,\gamma_{p^t}\in\Gamma^\ac$ be such that $\gamma_iv_1=v_i$.
Then we have the direct sum decompositions
\begin{equation}\label{eq:dec}
\Lambda_{\ac}=\bigoplus_{i=1}^{p^t}\gamma_i\bZ_p[[\Gamma_1^\ac]],\quad\quad
H^1_\pm(K_v,\Tc)=\bigoplus_{i=1}^{p^t}\gamma_iH^1_\pm(K_v,\mathbf{T}^\ac_1),
\end{equation}
where $\mathbf{T}^\ac:=T\otimes\Lambda_{\ac}$ equipped with the diagonal $G_K$-action similarly as before.

\begin{defn}\label{def:log-ac}
For every $v\mid p$ in $K$, define the maps
\[
{\rm Log}_\ac^\pm:H^1_\pm(K_v,\mathbf{T}_1^\ac)\longrightarrow\bZ_p[[\Gamma_1^\ac]]
\] 
by the relation
\[
x={\rm Log}_\ac^\pm(x)\cdot a^\pm
\]
for all $x\in H^1_\pm(K_v,\mathbf{T}_1^\ac)$. This naturally extends to a map  
${\rm Log}_\ac^\pm:H^1_\pm(K_v,\Tc)\longrightarrow\Lambda_\ac$ using $(\ref{eq:dec})$.
which does not depend on the choice of $\gamma_i$.
\end{defn}

The following result establishes the interpolation property satisfied by the map
${\rm Log}_\ac^+$ (the result for ${\rm Log}_\ac^-$ is entirely similar).

\begin{lem}\label{lem:interpolation}
Let $\phi:\Gamma^\ac\rightarrow\bC_p^\times$ be a finite order character of conductor $p^n$, with $n>0$ even.
If $x=\varprojlim_nx_n\in H^1_+(K_v,\Tc)$, then the following formulas hold:
\begin{align*}
\sum_{\tau\in\Gamma^\ac/p^n\Gamma^\ac}\phi(\tau)\log_{\hat{E}}(x_n^\tau)\cdot\tilde{\omega}_n^-(\phi)&=
\phi^{-1}\bigl({\rm Log}_\ac^+(x)\bigr)\cdot(-1)^{n/2}\sum_{\tau\in\Gamma^\ac/p^n\Gamma^\ac}\phi(\tau)
\log_{\hat{E}}(c_{n,n+a}^\tau),\\
\sum_{\tau\in\Gamma^\ac/p^n\Gamma^\ac}\phi(\tau)\log_{\hat{E}}(c_{n,n+a}^\tau)&=\frac{\mathfrak{g}(\phi)}{\phi(p^{n})}
\sum_{\tau\in\Gamma^\ac/p^n\Gamma^\ac}\phi(\tau)\;d_{n+a}^\tau,
\end{align*}
where $\mathfrak{g}(\phi)=\sum_{u\;{\rm mod}\;p^n}\phi(u)\zeta_{p^n}^u$ is the Gauss sum of $\phi$.
\end{lem}

\begin{proof}
The first equality follows directly from the definitions. On the other hand,
an immediate calculation using $(\ref{eq:def-c})$ and $(\ref{eq:def-cmn})$ (cf. \cite[Lemma~6.2]{wan-combined})
reveals that
\begin{align*}
\log_{\wh{E}}(c_{n,n+a})&=\sum_i(-1)^{i-1}d_{n+a}^{\varphi^{-(n+2i)}}\cdot p^i
+\sum_{0\leqslant 2k<n}(-1)^kd_{n+a}^{\varphi^{2k-n}}\cdot\frac{\zeta_{p^{n-2k}}-1}{p^k}.
\end{align*}
Thus we find that
\begin{align*}
\sum_{\tau\in\Gamma^\ac/p^n\Gamma^\ac}\phi(\tau)\log_{\hat{E}}(c_{n,n+a}^\tau)
&=\sum_{\tau\in\Gamma^\ac/p^n\Gamma^\ac}\phi(\tau)\;d_{n+a}^{\varphi^{-n}\tau}\cdot(\zeta_{p^n}^\tau-1)\\
&=\mathfrak{g}(\phi)\sum_{\tau\in\Gamma^\ac/p^n\Gamma^\ac}\phi(\tau)\;d_{n+a}^{\varphi^{-n}\tau}
=\frac{\mathfrak{g}(\phi)}{\phi(p^n)}\sum_{\tau\in\Gamma^\ac/p^n\Gamma^\ac}\phi(\tau)\;d_{n+a}^\tau.
\end{align*}
\end{proof}


\subsection{The two-variable plus/minus Selmer groups}\label{sec:selmer}



As in the preceding sections, let $\T$ denote the $p$-adic Tate module of $E$, and set $\V=\T\otimes_{\bZ_p}\bQ_p$ and $\A=\V/\T\simeq E[p^\infty]$. Let $\Sigma$ be a finite set of places of $K$
containing those dividing $Np\infty$, and let $\mathfrak{G}_{K,\Sigma}$ be
the Galois group of the maximal extension of $F$ unramified outside the places above $\Sigma$. Recall that $\Gamma_K$ is the Galois group of
the $\bZ_p^2$-extension $K_\infty/K$, and define the $\Lambda=\ro[[\Gamma_K]]$-modules
\[
\mathbf{T}:=\T\otimes_{\ro}\Lambda_{}(\can^{-1}),
\quad\quad\mathbf{A}:=\mathbf{T}\otimes_{\Lambda}{\rm Hom}_{\bZ_p}(\Lambda,\bQ_p/\bZ_p),
\]
where $\can:\Gamma_K\hookrightarrow\Lambda^\times$ is the map sending $\gamma\in\Gamma_K$
to the corresponding group-like element in $\Lambda^\times$. 
We shall also need to consider the modules $\mathbf{T}^{\rm ac}$, $\mathbf{A}^{\rm ac}$, $\mathbf{T}^{\rm cyc}$, and $\mathbf{A}^{\rm cyc}$, obtained by replacing $\Gamma_K$ in the preceding definitions by the Galois group $\Gamma^{\rm ac}$ and $\Gamma^{\rm cyc}$ of the anticyclotomic and the cyclotomic $\bZ_p$-extension of $K$, respectively. 

In the following definitions, we let $\mathbf{M}$ denote either of the modules $\mathbf{T}$, $\mathbf{T}^{\rm ac}$, $\mathbf{T}^{\rm cyc}$, or any of their specializations.


\begin{defn}\label{def:Sel}
	The $p$-relaxed Selmer group of $\mathbf{M}$ is
	\begin{equation}
	\mathfrak{Sel}^{\{p\}}(K,\mathbf{M}):={\rm ker}\Biggl\{H^1(\mathfrak{G}_{K,\Sigma},\mathbf{M})
	\longrightarrow\bigoplus_{v\in\Sigma\smallsetminus\{p\}}\frac{H^1(K_v,\mathbf{M})}{H^1_{\rm ur}(K_v,\mathbf{M})}\Biggr\},\nonumber
	\end{equation}
	where 
	\[
	H^1_{\rm ur}(K_v,\mathbf{M}):=
	{\rm ker}\{H^1(K_v,\mathbf{M})\longrightarrow H^1(K_v^{\rm ur},\mathbf{M})\}
	\]
	is the unramified local condition. 
\end{defn}

Our Selmer groups of interest in this paper are obtained from cutting the $p$-relaxed ones by various local conditions at the primes above $p$.

\begin{defn}
	For $\qq\in\{\pp,\overline{\pp}\}$ and $\mathscr{L}_\qq\in\{{\rm rel},\pm,{\rm str}\}$, set
	\[
	H^1_{\mathscr{L}_\qq}(K_\qq,\mathbf{M})=
	\left\{
	\begin{array}{ll}
	H^1(K_\qq,\mathbf{M})&\textrm{if $\mathscr{L}_\qq={\rm rel}$,}\\
	H^1_\pm(K_\pp,\mathbf{M})&\textrm{if $\mathscr{L}_\qq=\pm$,}\\
	\{0\}&\textrm{if $\mathscr{L}_\qq={\rm str}$,}	
\end{array}
	\right.
	\]
and for $\mathscr{L}=\{\mathscr{L}_\pp,\mathscr{L}_{\overline{\pp}}\}$, define
	\begin{equation}
	\mathfrak{Sel}^{\mathscr{L}}(K,\mathbf{M}):={\rm ker}\Biggr\{{\rm Sel}^{\{p\}}(K,\mathbf{M})
	\longrightarrow\bigoplus_{\qq\in\{\pp,\overline{\pp}\}}\frac{H^1(K_\qq,\mathbf{M})}{H_{\mathscr{L}_{\qq}}^1(K_\qq,\mathbf{M})}\Biggr\},\nonumber
	\end{equation}
	and similarly for ${\rm Sel}^{\mathscr{L}}(K,\mathbf{M})$
\end{defn}

Thus, for example, $\mathfrak{Sel}^{{\rm rel},{\rm str}}(K,\mathbf{M})$ is the submodule of $\mathfrak{Sel}^{\{p\}}(K,\mathbf{M})$ consisting of classes which are trivial at $\overline{\pp}$ (with no condition at $\pp$). Also, letting $\mathbf{W}$ denote either of the modules $\mathbf{A}$, $\mathbf{A}^{\rm ac}$, $\mathbf{A}^{\rm cyc}$, or any of their specilizations, we define $\mathfrak{Sel}^{\{p\}}(K,\mathbf{W})$ and ${\rm Sel}^{\{p\}}(K,\mathbf{W})$ just as in Definition~\ref{def:Sel}, and set
\begin{equation}
\mathfrak{Sel}^{\mathscr{L}}(K,\mathbf{W}):={\rm ker}\Biggr\{\mathfrak{Sel}^{\{p\}}(K,\mathbf{W})
\longrightarrow\bigoplus_{\qq\in\{\pp,\overline{\pp}\}}\frac{H^1(K_\qq,\mathbf{W})}{H_{\mathscr{L}_{\qq}}^1(K_\qq,\mathbf{M})^\perp}\Biggr\},\nonumber
\end{equation}
where $H_{\mathscr{L}_{\qq}}^1(K_\qq,\mathbf{M})^\perp$ is the orthogonal complement of $H_{\mathscr{L}_{\qq}}^1(K_\qq,\mathbf{M})$ under local Tate duality. 
Finally, we let 
\[
\mathfrak{X}^{\mathscr{L}}(K,\mathbf{A}):={\rm Hom}_{\bZ_p}(\mathfrak{Sel}^{\mathscr{L}}(K,\mathbf{A}),\bQ_p/\bZ_p)
\]
be the Pontrjagin dual of $\mathfrak{Sel}^{\mathscr{L}}(K,\mathbf{A})$, and similarly define $\mathfrak{X}^{\mathscr{L}}(K,\mathbf{A}^{\rm ac})$ and $\mathfrak{X}^{\mathscr{L}}(K,\mathbf{A}^{\rm cyc})$. 


\subsubsection*{Anticyclotomic Selmer groups} 

Let $\iota:\Lambda_\ac\rightarrow\Lambda_\ac$ the involution given by
$\gamma\mapsto\gamma^{-1}$ on group-like elements, and for any $\Lambda_\ac$-module $M$,
let $M^\iota$ denote the underlying module $M$ with the $\Lambda_\ac$-module
structure given by $\Lambda_\ac\xrightarrow{\iota}\Lambda_\ac\rightarrow{\rm End}(M)$. Denote by $M_{\rm tors}$ the $\Lambda_\ac$-torsion submodule of $M$.

\begin{lem}\label{lem:str-rel}
For every $\varepsilon\in\{\pm\}$ we have 
${\rm rank}_{\Lambda_\ac}\mathfrak{X}^{\varepsilon,{\rm rel}}(K,\Ac)
=1+{\rm rank}_{\Lambda_\ac}\mathfrak{X}^{\varepsilon,{\rm str}}(K,\Ac)$,  
and
\[
{\rm Char}_{\Lambda_\ac}(X^{\varepsilon,{\rm rel}}(K,\Ac)_{\rm tors})
={\rm Char}_{\Lambda_\ac}(X^{\varepsilon,{\rm str}}(K,\Ac)_{\rm tors})
\]
as ideals in $\Lambda_\ac$.
\end{lem}

\begin{proof}
We shall adapt the arguments in \cite[\S{1.2}]{AHsplit}. 
By Lemma~3.5.3 and Theorem~4.1.13 of \cite{MR-KS}, for every continuous character $\psi:\Gamma^\ac\rightarrow L^\times$
with values in some finite extension $L/\bQ_p$ with ring of integers $\mathfrak{O}_L$,
there is a non-canonical isomorphism
\begin{equation}\label{eq:MR}
H^1_{\eps,{\rm rel}}(K,\Ac(\psi))[p^i]\simeq(L/\mathfrak{O}_L)^r[p^i]
\oplus H^1_{\eps,{\rm str}}(K,\Ac(\psi^{-1}))[p^i]
\end{equation}
for all positive $i$. Here $H^1_{\eps,{\rm rel}}(K,\Ac(\psi))\subset {\rm Sel}^{\eps,{\rm rel}}(K,\Ac(\psi))$
is the generalized Selmer group consisting of classes whose restriction at $\pp$ lies in 
$H^1(K_\pp,\Ac(\psi))_{\rm div}$, while $H^1_{\eps,{\rm str}}(K,\Ac(\psi^{-1}))$
is the same as ${\rm Sel}^{\eps,{\rm str}}(K,\Ac(\psi^{-1}))$,
and $r$ is the \emph{core rank} (see \cite[Def.~4.1.11]{MR-KS})
of the Selmer conditions defining $H^1_{\eps,{\rm rel}}(K,\Ac(\psi))$,
which by \cite[Thm.~2.18]{DDT} it is given by the quantity
\begin{equation}\label{eq:DDT}
{\rm corank}_{\mathfrak{O}_L}H^1_\eps(K_\pp,\Ac(\psi))
+{\rm corank}_{\mathfrak{O}_L}H^1(K_{\overline{\pp}},\Ac(\psi))
-{\rm corank}_{\mathfrak{O}_L}H^0(K_{w},\Ac(\psi)),
\end{equation}
where $w$ denotes the infinite place of $K$. By the local Euler characteristic formula,
the first two terms in $(\ref{eq:DDT})$ are equal to $1$ and $2$, respectively, while the third one clearly equals $2$.
Thus $r=1$ in $(\ref{eq:MR})$ and letting $i\to\infty$ we conclude that
\begin{equation}\label{eq:MR2}
H^1_{\eps,{\rm rel}}(K,\Ac(\psi))\simeq(L/\mathfrak{O}_L)
\oplus H^1_{\eps,{\rm str}}(K,\Ac(\psi^{-1})).
\end{equation}
Now, it is easy to show that the natural restriction maps
\begin{align*}
H^1_{\eps,{\rm rel}}(K,\Ac(\psi))
&\longrightarrow\mathfrak{Sel}^{\eps,{\rm rel}}(K,\Ac)(\psi)^{\Gamma^\ac}\\
H^1_{\eps,{\rm str}}(K,\Ac(\psi^{-1}))
&\longrightarrow\mathfrak{Sel}^{\eps,{\rm str}}(K,\Ac)(\psi^{-1})^{\Gamma^\ac}
\end{align*}
are injective with finite bounded cokernel as $\psi$ varies (cf. \cite[Lem.~1.2.4]{AHsplit}),
and since
\[
\mathfrak{Sel}^{\eps,{\rm str}}(K,\Ac)(\psi^{-1})^{\Gamma^\ac}
\simeq\mathfrak{Sel}^{\eps,{\rm str}}(K,\Ac)(\psi)^{\Gamma^\ac}
\]
by the action of complex conjugation, the result follows from $(\ref{eq:MR2})$
by the same argument as in \cite[Lem.~1.2.6]{AHsplit}, proceeding as in \cite[Thm.~2.2.1]{howard-PhD-I} to handle the prime $p\Lambda_{\ac}$.
\end{proof} 

\section{Beilinson--Flach classes}

\subsection{The plus/minus Beilinson--Flach classes}\label{subsec:pm-BF}

In this section, building on the work of Loeffer--Zerbes \cite{LZ2},
we show the existence of certain plus/minus Beilinson--Flach classes $\mathcal{BF}^\pm$ which map
to the plus/minus $p$-adic $L$-functions of $\S\ref{subsec:Loeffler}$
under the plus/minus Coleman maps.
We maintain the set-up introduced in Section~\ref{subsec:local-p}, and let
$f=\sum_{n=1}^\infty a_nq^n\in S_2(\Gamma_0(N))$ be the normalized newform associated with $E$.
In particular, $a_p=0$.

Recall the $\bZ_p[[\Gamma_1]]$-linear maps ${\rm Col}^\pm$ introduced in $(\ref{def:Col-pm})$,
and extend them (using the analogue of the decomposition $(\ref{eq:dec})$ with $\Gamma$ in place of $\Gamma^\ac$)
to $\Lambda$-linear isomorphisms
\[
{\rm Col}^\pm:\frac{H^1(K_v,\mathbf{T})}{H^1_{\pm}(K_v,\mathbf{T})}\longrightarrow\Lambda
\]
for each prime $v$ above $p$. On the other hand, the $p$-adic $L$-function $L_\pp(f/K)$ of Theorem~\ref{thm:wan} we defined as an element in $\Lambda_{\unr}$,
but (as we show in $\S\ref{subsec:HP}$) the corresponding principal ideal in $\Lambda_{\unr}$ can be generated by an element $\mathcal{L}_\pp(f/K)\in\Lambda$.


In the following, we let $\gamma_\ac\in\Gamma^\ac$ be a topological generator, and let $P\subseteq\Lambda$ be the pullback of the augmentation ideal $(\gamma_\ac-1)\subseteq\Lambda_\ac$.

\begin{thm}\label{thm:BF-ERL}
For every $\eps\in\{\pm\}$ there exist an element $\mathcal{BF}^\eps\in{\rm Sel}^{\eps,{\rm rel}}(K,\mathbf{T})$
such that
\[
{\rm Col}^\eps({\rm res}_{\overline{\pp}}(\mathcal{BF}^\eps))=u\cdot h^\eps\cdot L_p^{\eps,\eps}(f/K),
\quad{\rm Log}^\eps({\rm res}_{\pp}(\mathcal{BF}^\eps))=h^\eps\cdot \mathcal{L}_\pp(f/K),
\]
for some nonzero $u\in\Lambda[1/P]$ and $h^\eps\in\Lambda$. 
\end{thm}

\begin{proof}
This is shown in \cite[\S{7.3}]{wan-combined}, where it is deduced
from the work of Loeffler--Zerbes on Beilinson--Flach classes in Coleman families
and their explicit reciprocity laws \cite{LZ-Coleman}. (Note that \cite{wan-combined} only deals with $\eps=+1$, but the case $\eps=-1$ is done completely analogously.) The element $h^\eps$ is needed to establish integrality of the class $\mathcal{BF}^\eps$, while $u$ is the ratio between two periods attached to a certain
Hida family of CM forms, whose integrality (up to powers of the ``exceptional prime'' $P$) is shown in
\cite[\S{8.1}]{wan-combined} building on work of Rubin \cite{rubin-IMC} and Hida--Tilouine \cite{HT-117}. Thus it remains to take care of the ratio between Hida's canonical period $\Omega_f^{\tt Hida}$ in $(\ref{eq:Hida-period})$ and the Petersson inner product period used in \cite{LZ-Coleman}, which we do with the following comparison. (This may be well-known to experts, but
we provide the details for the convenience of the reader.) 

Letting $D_{\rm FL}$ denote the Fontaine--Laffaille functor, we have
\[
D_{\rm FL}(H^1_{\mathrm{et}}(X_0(N),\bZ_p))=H^1(X_0(N),\Omega^\bullet_{X_0(N)/\bZ_p})
\]
(see \cite[\S{6.10}]{LLZ}). The Galois representation $T\simeq T_f^*$ is given by
$H^1_{\mathrm{et}}(X_0(N),\bZ_p)[\lambda_f]$, where $[\lambda_f]$ denotes the maximal submodule of $H^1_{\mathrm{et}}(X_0(N),\bZ_p)$ on which
the Hecke algebra $\mathbb{T}_0(N)$ acts with the same eigenvalues as in $f$,
and hence
\[
D_{\rm FL}(T)=H^1(X_0(N),\Omega^\bullet_{X_0(N)/\bZ_p})[\lambda_f].
\]

On the other hand, we have an exact sequence of localized Hecke modules
\[
H^0(X_0(N),\Omega^1_{X_0(N)/\bZ_p})_{\mathfrak{m}_f}\hookrightarrow
H^1(X_0(N),\Omega^\bullet_{X_0(N)/\bZ_p})_{\mathfrak{m}_f}
\twoheadrightarrow H^1(X_0(N),\mathcal{O}_{X_0(N)})_{\mathfrak{m}_f}.
\]
The last term is free of rank one over $\mathbb{T}_0(N)_{\mathfrak{m}_f}$,
while the first is isomorphic to $S(X_0(N),\bZ_p)_{\mathfrak{m}_f}$. 
Unravelling the definitions, we see that the class $\eta_f\in H^1(X_0(N),\bZ_p)_{\mathfrak{m}_f}\otimes_{\bZ_p}\bQ_p$ constructed
in \cite[\S{6.10}]{LLZ} corresponds to the projector to the $f$-component
under the identification of $H^1(X_0(N),\bZ_p)_{\mathfrak{m}_f}\otimes_{\bZ_p}\bQ_p$
with $\mathbb{T}_0(N)_{\mathfrak{m}_f}$, while we have an isomorphism
\[
D_{\mathrm{FL}}(T)/\mathrm{Fil}^0D_{\rm FL}(T)\simeq H^1(X_0(N),\mathcal{O}_{X_0(N)})_{\mathfrak{m}_f}[\lambda_f].
\]
Thus the ratio of a generator of $D_{\mathrm{FL}}(T)/\mathrm{Fil}^0D_{\mathrm{FL}}(T)$
over $\eta_f$ is, by definition, the congruence number $c_f$ of $f$, from where the desired comparison follows by $(\ref{eq:Hida-period})$.
\end{proof}

\begin{rem}\label{rem:h-ac}
By the construction in \cite[\S{7.3}]{wan-combined}, the element $h^\eps$ divides the half-logarithm $\log_{\pp}^\eps$.
Since the latter does not vanish identically along the anticyclotomic line, the same is true for $h^\eps$.
\end{rem}


%

\subsection{Two-variable main conjectures}\label{sec:ES}

As shown below, by Theorem~\ref{thm:BF-ERL} the following three different variants of Iwasawa's main conjecture
for elliptic curves $E/\bQ$ at supersingular primes $p>3$ base-changed to an imaginary quadratic field $K$ in which $p=\pp\overline{\pp}$ splits are essentially equivalent:

\begin{enumerate}
\item{} The main conjecture `without $p$-adic zeta functions' for the two-variable plus/minus Selmer groups.
\item{} The Iwasawa--Greenberg main conjecture for $L_\pp(f/K)$.
\item{} The equal-sign cases of Kim's two-variable main conjectures \cite{kim-CJM}.
\end{enumerate}
More precisely, we have the following:

\begin{thm}\label{thm:2-varIMC}
Let $\eps\in\{\pm\}$. The following two statements are equivalent:
\begin{enumerate}
\item{} $\mathfrak{X}^{{\rm rel},{\rm str}}(K,\mathbf{A})$ is $\Lambda$-torsion, and
\[
{\rm Char}_{\Lambda}(\mathfrak{X}^{{\rm rel},{\rm str}}(K,\mathbf{A}))=(\mathcal{L}_\pp(f/K))
\]
as ideals in $\Lambda$. 
\item{} $\mathfrak{X}^{\eps,{\rm str}}(K,\mathbf{A})$ is $\Lambda$-torsion,
$\mathfrak{Sel}^{\eps,{\rm rel}}(K,\mathbf{T})$ has $\Lambda$-rank $1$, and
\[
{\rm Char}_{\Lambda}(\mathfrak{X}^{\eps,{\rm str}}(K,\mathbf{A}))\cdot\mathcal{H}^\eps=
{\rm Char}_\Lambda\bigg(\frac{\mathfrak{Sel}^{\eps,{\rm rel}}(K,\mathbf{T})}{\Lambda\cdot\mathcal{BF}^\eps}\biggr),
\]
where $\mathcal{H}^\eps\subseteq\Lambda$
is the ideal generated by the element $h^\eps$ in Theorem~\ref{thm:BF-ERL}.
%
\end{enumerate}
Moreover, $\mathfrak{X}^{\eps,\eps}(K,\mathbf{A})$ is $\Lambda$-torsion, and if
the above two statements hold, then we have the divisibility
\[
{\rm Char}_{\Lambda}(\mathfrak{X}^{\eps,\eps}(K,\mathbf{A}))\subseteq (L_p^{\eps,\eps}(f/K))
\]
in $\Lambda[1/P]$.
\end{thm}

\begin{proof}
This is essentially shown in \cite[\S{8.1}]{wan-combined}. 
For the convenience of the reader, we briefly recall the argument.
Poitou--Tate global duality
gives rise to the exact sequences
\begin{equation}\label{eq:ES-1b}
0\longrightarrow\mathfrak{Sel}^{\eps,{\rm rel}}(K,\mathbf{T})\longrightarrow
H^1_{\eps}(K_{\pp},\mathbf{T})\longrightarrow
\mathfrak{X}^{{\rm rel},{\rm str}}(K,\mathbf{A})
\longrightarrow \mathfrak{X}^{\eps,{\rm str}}(K,\mathbf{A})\longrightarrow 0,
\end{equation}
\begin{equation}\label{eq:ES-1a}
0\longrightarrow\mathfrak{Sel}^{\eps,{\rm rel}}(K,\mathbf{T})\longrightarrow
\frac{H^1(K_{\overline{\pp}},\mathbf{T})}{H^1_{\eps}(K_{\overline{\pp}},\mathbf{T})}
\longrightarrow \mathfrak{X}^{\eps,\eps}(K,\mathbf{A})
\longrightarrow \mathfrak{X}^{\eps,{\rm str}}(K,\mathbf{A})\longrightarrow 0,
\end{equation}
where exactness on the leftmost terms relies on the nonvanishing
of 
$L_p^{\eps,\eps}(f/K)$ and $L_\pp(f/K)$. 

By control theorem \cite[Prop.~8.7]{wan-combined}, Kobayashi's result \cite[Thm.~1.2]{kobayashi-152} implies that $\mathfrak{X}^{\eps,\eps}(K,\mathbf{A})$ is $\Lambda$-torsion. By $(\ref{eq:ES-1a})$, it follows that
$\mathfrak{X}^{\eps,{\rm str}}(K,\mathbf{A})$ is $\Lambda$-torsion and $\mathfrak{Sel}^{\eps,{\rm rel}}(K,\mathbf{T})$
has $\Lambda$-rank one, and by $(\ref{eq:ES-1a})$, that $\mathfrak{X}^{{\rm rel},{\rm str}}(K,\mathbf{A})$ is $\Lambda$-torsion.
Finally, \cite[Cor.~7.9]{wan-combined} and 
Theorem~\ref{thm:BF-ERL} yield the following exact sequences from the above:
\[
0\longrightarrow\frac{\mathfrak{Sel}^{\eps,{\rm rel}}(K,\mathbf{T})}{\Lambda\cdot\mathcal{BF}^\eps}\longrightarrow
\frac{\Lambda}{\mathcal{H}^\eps\cdot(\mathcal{L}_\pp(f/K))}\longrightarrow
\mathfrak{X}^{{\rm rel},{\rm str}}(K,\mathbf{A})
\longrightarrow \mathfrak{X}^{\eps,{\rm str}}(K,\mathbf{A})\longrightarrow 0,
\]
\[
0\longrightarrow\frac{\mathfrak{Sel}^{\eps,{\rm rel}}(K,\mathbf{T})}{\Lambda\cdot\mathcal{BF}^\eps}\longrightarrow
\frac{\Lambda}{\mathcal{H}^\eps\cdot\mathcal{U}\cdot(L_p^{\eps,\eps}(f/K))}
\longrightarrow \mathfrak{X}^{\eps,\eps}(K,\mathbf{A})
\longrightarrow \mathfrak{X}^{\eps,{\rm str}}(K,\mathbf{A})\longrightarrow 0,
\]
where $\mathcal{U}\subset\Lambda[1/P]$ is the ideal generated by the element $u$ in Theorem~\ref{thm:BF-ERL}.
By the multiplicativity of characteristic ideals along exact sequences, the result follows.
\end{proof}




\begin{cor}\label{thm:str}
If any of the equivalent statements in Theorem~\ref{thm:2-varIMC} holds, then
\[
{\rm Char}_{\Lambda_\ac}(\mathfrak{X}^{\eps,{\rm str}}(K,\Ac))\cdot\mathcal{H}^\eps_\ac
={\rm Char}_{\Lambda_\ac}\biggl(\frac{\mathfrak{Sel}^{\eps,{\rm rel}}(K,\Tc)}{\Lambda^\ac\cdot\mathcal{BF}_\ac^\eps}\biggr),
\]
where $\mathcal{BF}_\ac^\eps$ is the image of $\mathcal{BF}^\eps$ under the
projection $\mathfrak{Sel}^{\eps,{\rm rel}}(K,\mathbf{T})\rightarrow\mathfrak{Sel}^{\eps,{\rm rel}}(K,\Tc)$ and
$\mathcal{H}_\ac^\eps$ is the image of $\mathcal{H}^\eps$ in  $\Lambda_{\ac}$.
\end{cor}

\begin{proof}
Of course, this follows from descending part (2) of Theorem~\ref{thm:2-varIMC} from $K_\infty$ to $K_\infty^\ac$.
Let $\gamma_{\rm cyc}\in\Gamma^{\rm cyc}$ be a topological generator, and let $I^{\rm cyc}:=(\gamma_{\rm cyc}-1)\Lambda\subset\Lambda$. As in  \cite[Prop.~3.9]{SU} (with
the roles of the cyclotomic and anticyclotomic $\bZ_p$-extensions reversed) we have
\[
\mathfrak{X}^{\eps,{\rm str}}(K,\mathbf{A})/I^{\rm cyc}\mathfrak{X}^{\pm,{\rm str}}(K,\mathbf{A})\simeq
\mathfrak{X}^{\eps,{\rm str}}(K,\Ac),
\]
and by \cite[Lem.~6.2(ii)]{rubin-IMC} it follows that
\begin{equation}\label{eq:3.9}
{\rm Char}_{\Lambda_\ac}(\mathfrak{X}^{\eps,{\rm str}}(K,\Ac))
={\rm Char}_{\Lambda}(\mathfrak{X}^{\eps,{\rm str}}(K,\mathbf{A}))\cdot\mathfrak{D},
\end{equation}
where $\mathfrak{D}:={\rm Char}_{\Lambda_\ac}(\mathfrak{X}^{\eps,{\rm str}}(K,\mathbf{A})[I^{\rm cyc}])$.
On the other hand, set
\[
Z(K_\infty):=\mathfrak{Sel}^{\eps,{\rm rel}}(K,\mathbf{T})/(\mathcal{BF}^\eps),
\quad
Z(K_\infty^\ac):=\mathfrak{Sel}^{\eps,{\rm rel}}(K,\Tc)/(\mathcal{BF}^\eps_{\ac}).
\]
Using the fact that $I^{\rm cyc}$ is principal,
an application of snake's lemma yields the exactness of
\begin{equation}\label{eq:snake}
\mathfrak{Sel}^{\eps,{\rm rel}}(K,\mathbf{T})[I^{\rm cyc}]\longrightarrow
Z(K_\infty)[I^{\rm cyc}]\longrightarrow(\mathcal{BF}^\eps)/I^{\rm cyc}(\mathcal{BF}^\eps).
\end{equation}
Arguing as in the proof of \cite[Prop.~2.4.15]{AHsplit} we see that the natural
$\Lambda_\ac$-module map
\[
Z(K_\infty)/I^{\rm cyc}Z(K_\infty)\longrightarrow Z(K_\infty^\ac)
\]
is injective with cokernel having characteristic ideal $\mathfrak{D}$, and hence
\begin{equation}\label{eq:2.4.5}
Ch_{\Lambda^\ac}(Z(K_\infty^\ac))=Ch_{\Lambda_\ac}(Z(K_\infty)/I^{\rm cyc}Z(K_\infty))\cdot\mathfrak{D}.
\end{equation}

Since $H^1(K,\mathbf{T})$ has trivial $\Lambda$-torsion, the leftmost term in $(\ref{eq:snake})$ vanishes;
on the other hand, the rightmost one is clearly torsion-free, and hence $Z(K_\infty)[I^{\rm cyc}]$ is torsion-free.
Since \cite[Lem.~6.2(i)]{rubin-IMC} and equality $(\ref{eq:2.4.5})$ imply
that $Z(K_\infty)[I^{\rm cyc}]$ is also a torsion $\Lambda_\ac$-module
(using the nonvanishing of the terms in that equality), we conclude that $Z(K_\infty)[I^{\rm cyc}]=\{0\}$,
and by \cite[Lem.~6.2(ii)]{rubin-IMC} it follows that
\begin{equation}\label{eq:6.2}
{\rm Char}_{\Lambda}(Z(K_\infty))\cdot\Lambda_\ac
={\rm Char}_{\Lambda_\ac}(Z(K_\infty)/I^{\rm cyc}Z(K_\infty)).
\end{equation}
Combined with $(\ref{eq:3.9})$, 
we thus arrive at
\begin{equation}
\begin{split}\label{eq:Z}
{\rm Char}_{\Lambda_\ac}(\mathfrak{X}^{\eps,{\rm str}}(K,\Ac))
&={\rm Char}_{\Lambda}(\mathfrak{X}^{\eps,{\rm str}}(K,\mathbf{A}))\cdot\mathfrak{D}\\
&=\mathcal{H}^\eps\cdot {\rm Char}_{\Lambda}(Z(K_\infty))\cdot\mathfrak{D}\\
&=\mathcal{H}^\eps_\ac\cdot {\rm Char}_{\Lambda_\ac}(Z(K_\infty^\ac)),
\end{split}\nonumber
\end{equation}
using $(\ref{eq:2.4.5})$ and $(\ref{eq:6.2})$ for the last equality. This completes the proof.
\end{proof}

\subsection{Rubin's height formula}\label{subsec:rubin}

We keep the notations introduced in $\S\ref{sec:anti-L}$, assume that the generalized Heegner hypothesiss ${\rm(Heeg)}$
in that section holds,
and still denote by
$L_{p}^{\bullet,\circ}(f/K)$ the image of the $p$-adic $L$-functions
$L_p^{\bullet,\circ}(f/K)$ of Proposition~\ref{prop:loeffler+-} 
under the projection
\[
\ro[[H_{\frakf p^\infty}]]\otimes_{\ro}\bQ_p\longrightarrow
\Lambda\otimes_{\ro}\bQ_p.
\]
Let $\gamma_{\rm cyc}\in\Gamma^{\rm cyc}$ be a topological generator, and using the identification
$\Lambda\simeq\Lambda^\ac[[\Gamma^{\rm cyc}]]$ expand
\begin{equation}\label{eq:expand}
L^{\bullet,\circ}_{p}(f/K)=L_{p,0}^{\bullet,\circ}(f/K)+L_{p,1}^{\bullet,\circ}(f/K)(\gamma_{\rm cyc}-1)+\cdots 
\end{equation}
as a power series in $\gamma_{\rm cyc}-1$ with coefficients in $\Lambda_\ac\otimes_{\bZ_p}\bQ_p$. 
Thus $L_{p,0}^{\bullet,\circ}(f/K)$ is identified with the anticyclotomic projection $L^{\bullet,\circ}_{p,\ac}(f/K)$, and so
\[
L_{p,0}^{\eps,\eps}(f/K)=0
\]
for each $\eps\in\{\pm\}$ by Corollary~\ref{cor:signed-0}. 
In particular, by Theorem~\ref{thm:BF-ERL} and the injectivity of the map ${\rm Col}^\eps$, it follows that the classes $\mathcal{BF}^\eps$ have images $\mathcal{BF}_{\ac}^\eps$ under the projection $H^1(K,\mathbf{T})\rightarrow H^1(K,\Tc)$ landing in $\mathfrak{Sel}^{\eps,\eps}(K,\Tc)\subset\mathfrak{Sel}^{\eps,{\rm rel}}(K,\Tc)$.

In the following, let $h_0^\eps\in\Lambda_\ac$
and $u_0\in\Lambda_{\ac}[1/P]$
to be the constant term in the expansion of the elements $h^\eps$ and $u$ in Theorem~\ref{thm:BF-ERL} as a power series in $\gamma_{\rm cyc}-1$, so that
$\mathcal{H}^\eps_\ac=(h^\eps_0)$ in the notations of Corollary~\ref{thm:str}. Also, we let $L_n=K_n^\ac K_\infty^{\rm cyc}$,
and think of $\mathcal{BF}^\eps\in H^1(K,\mathbf{T})\simeq H^1_{\rm Iw}(K_\infty,T)$
as a compatible system of classes $\mathcal{BF}^\eps_{{\rm cyc},n}\in H^1_{\rm Iw}(L_n,T)$.

\begin{lem}\label{lem:3.1.1}
For every $n$ there is a unique element
\[
\beta_n^\eps\in\frac{H^1_{\rm Iw}(L_{n,\overline{\pp}},T)}{H^1_{{\rm Iw},\eps}(L_{n,\overline{\pp}},T)}\otimes_{\ro}\bQ_p
\]
such that
\[
(\gamma^{\rm cyc}-1)\beta_n^\eps={\rm loc}_{\overline{\frakp}}(\mathcal{BF}^\eps_{{\rm cyc},n}).
\]
Letting $\beta_n^\eps(\mathds{1})$ be the image of $\beta_n^\eps$ in
$H^1(K_{n,\overline{\pp}}^\ac,T)/H_{\eps}^1(K_{n,\overline{\pp}}^\ac,T)[1/p]$,
the elements $\beta_n^\eps(\mathds{1})$ 
define an element $\beta_\infty^\eps(\mathds{1})\in H^1(K_{\overline{\pp}},\Tc)/H^1_\eps(K_{\overline{\pp}},\Tc)[1/p]$, 
and the map ${\rm Col}^{\eps}$ 
yields an identification
\[
\frac{H^1_{}(K_{\overline{\pp}},\Tc)}{H^1_{\eps}(K_{\overline{\pp}},\Tc)}
\otimes_{\ro}\bQ_p\simeq\Lambda^\ac\otimes_{\ro}\bQ_p
\]
sending $\beta_\infty^\eps(\mathds{1})$ to 
the product $u_0\cdot h_0^\eps\cdot L_{p,1}^{\eps,\eps}(f/K)$. 
\end{lem}

\begin{proof}
Since $u_0\cdot h_0^\eps\cdot L_{p,1}^{\eps,\eps}(f/K)$ is clearly the coefficient in the linear term of the expansion of $u\cdot h^\eps\cdot L_p^{\eps,\eps}(f/K)$ as a power series in $\gamma_{\rm cyc}-1$, the result follows
from the definition of $\beta^\eps_\infty(\mathds{1})$ and 
Theorem~\ref{thm:BF-ERL}.
\end{proof}

Let $\cI\subseteq\ro[[\Gamma^{\rm cyc}]]$ be the augmentation ideal,
and set $\mathcal{J}=\cI/\cI^2$.

\begin{thm}\label{thm:rubin-ht}
For every $n$ there is a canonical (up to sign)
$p$-adic height pairing
\[
\langle\;,\;\rangle_{K_n^\ac}^{\rm cyc}:\mathfrak{Sel}^{\eps,\eps}(K_n^\ac,T)\times\mathfrak{Sel}^{\eps,\eps}(K_n^\ac,T)
\longrightarrow p^{-k}\bZ_p\otimes_{\bZ_p}\mathcal{J}
\]
for some $k\in\bZ_{\geqslant 0}$ independent of $n$,
such that for every $b\in\mathfrak{Sel}^{\eps,\eps}(K_n^\ac,T)$, we have
\begin{equation}\label{eq:rubin-ht}
\langle\mathcal{BF}^\eps_{{\rm cyc},n}(\mathds{1}),b\rangle^{\rm cyc}_{K_n^\ac}
=(\beta^\eps_n(\mathds{1}),{\rm loc}_{\overline\pp}(b))_n\otimes(\gamma_{\rm cyc}-1),
\end{equation}
where $(\;,\;)_n$ is the $\bQ_p$-linear extension of the local Tate pairing
\[
\frac{H^1(K_{n,\overline{\pp}}^\ac,T)}{H^1_{\eps}(K_{n,\overline{\pp}}^\ac,T)}
\times H^1_{\eps}(K_{n,\overline\pp}^\ac,T)\longrightarrow\bZ_p.
\]
\end{thm}

\begin{proof}
Since by \cite[Prop.~4.11]{kim-parity} the 
local conditions defining $\mathfrak{Sel}^{\eps,\eps}(K_n^{\ac},T)$ at the primes $v\mid p$ are their own orthogonal complement
under the local Tate pairing, the construction of the cyclotomic $p$-adic height pairings $\langle\;,\rangle_{K_n^\ac}^{\rm cyc}$
can be deduced from \cite[Thm.~1.11]{howard-derived}.
The $p$-adic height formula $(\ref{eq:rubin-ht})$
then follows from \cite[Thm.~2.5(c)]{howard-derived}.
\end{proof}

\section{Heegner points}\label{sec:HP}
Let $E/\bQ$ be an elliptic curve, let
$f=\sum_{n=1}^\infty a_nq^n\in S_2(\Gamma_0(N))$ be the associated newform, and assume that
$p>3$ is a prime of good supersingular reduction for $E$ (so $a_p=0$). 
Let $K/\bQ$ be an imaginary quadratic field in which $p=\pp\overline{\pp}$ splits.
Throughout this section, we assume that the pair $(f,K)$ satisfies the
generalized Heegner hypothesis (Heeg) introduced in $\S\ref{sec:anti-L}$.

\subsection{The plus/minus Heegner classes}\label{subsec:HP}

%

Let $X_{N^+,N^-}$ be the Shimura curve (with the cusps added if $N^-=1$)
over $\bQ$ attached to a quaternion algebra $B/\bQ$ of discriminant $N^-$ and an Eichler order $R\subset\cO_B$ of level $N^+$.
We embed $X_{N^+,N^-}$ into its Jacobian $J_{N^+,N^-}$ by choosing
an auxiliary prime $\ell\nmid Np$ and defining
\[
\iota_{\ell}:X_{N^+,N^-}\longrightarrow J_{N^+,N^-}
\]
by $x\mapsto(T_{\ell}-\ell-1)[x]$, where $T_{\ell}$ is the usual Hecke correspondence on $X_{N^+,N^-}$,
and $[x]\in{\rm Div}(X_{N^+,N^-})$ is the divisor class of $x\in X_{N^+,N^-}$. 

Recall that we let $K[m]$ denote the ring class field of $K$ of conductor $m$.

\begin{prop}\label{prop:HP}
For every positive integer $m$ prime to $N$, there are Heegner points $P[m]\in E(K[m])$ such that
\[
{\rm tr}^{K[mp^{k+2}]}_{K[mp^{k+1}]}(P[mp^{k+2}])=
a_pP[mp^{k+1}]-P[mp^k]
\]
for all $k\geqslant 0$.
\end{prop}

\begin{proof}
This is standard: after fixing a modular parametrization
\[
\pi:J_{N^+,N^-}\longrightarrow E,
\] 	
a system of points as in the Proposition is obtained by letting $P[m]$ be the image of CM points $h[m]\in X_{N^+,N^-}(K[m])$ (see e.g. \cite[Prop.~1.2.1]{howard-PhD-II}) under the composite map $\pi\circ\iota_{\ell}$. 
\end{proof}

Since the $G_{\bQ}$-representation $E[p]\simeq\bar{\rho}_f$ is irreducible by \cite{Edi}, we may choose the above prime $\ell$ so that $a_{\ell}-\ell-1$ is a unit in $\ro^\times$, and we then let
\[
z[m]\in H^1(K[m],T)
\]
be the image of $P[m]\otimes(a_{\ell}-\ell-1)^{-1}$
under the Kummer map $E(K[m])\otimes\bZ_p\rightarrow H^1(K[m],T)$. (Thus $z[m]$ is independent of the choice of $\ell$.) 

The anticyclotomic $\bZ_p$-extension $K_\infty^\ac/K$ is contained in $\widetilde{K}_\infty=\bigcup_{k\geqslant 0}K[p^k]$,
and ${\rm Gal}(\widetilde{K}_\infty/K)$ is isomorphic to $\Gamma^\ac\times\Delta$, with
$\Delta$ finite. For every positive integer $S$ coprime to $Np$ and every $n\geqslant 0$,
we let $K_n^\ac[S]$ denote the compositum $K_n^\ac K[S]$, and set 
\[
z_n[S]:={\rm cor}^{K[Sp^{k(n)}]}_{K_n^\ac[S]}(z[Sp^{k(n)}])
\]
where $k(n):=\min\{k\;:\;K_n^\ac\subset K[p^k]\}$. Letting ${\rm cor}^{n+1}_{n}$ be corestriction map for the extension
$K_{n+1}^\ac[S]/K_{n}^\ac[S]$, it follows from the norm-compatibility in Proposition~\ref{prop:HP} that
\begin{equation}\label{(10)}
{\rm cor}^{n+1}_{n}(z_{n+1}[S])=-z_{n-1}[S],
\end{equation}
since $a_p=0$.


\begin{lem}\label{lem:UN}
The classes $z_n[S]$ lie in the image of the natural map
\begin{equation}\label{eq:proj}
H^1(K[S],\Tc)\longrightarrow H^1(K_n^\ac[S],T).\nonumber
\end{equation}
\end{lem}

\begin{proof}
The obvious long exact sequence shows that the cokernel of the map in the statement is controlled by
\begin{equation}\label{eq:n}
H^2(K[S],\Tc)[\omega_n].
\end{equation}
Since $H^2(K[S],\Tc)$ is finitely generated over $\Lambda_\ac$, the module $(\ref{eq:n})$ stabilizes for $n\gg 0$. On the other hand, from the norm-relation $(\ref{(10)})$ we immediately see that
\[
\frac{\omega_{n'}^\epsilon}{\omega_n^\epsilon}z_{n'}[S]=\pm z_n[S]
\]
for all $n'>n$ with $n'\equiv n\pmod{2}$, where $\epsilon=(-1)^n$.
Letting $n'\to\infty$, this shows that $z_n[S]$ must have zero image
in $(\ref{eq:n})$, hence the result.
\end{proof}



The next result will be a key ingredient in our construction of  plus/minus Heegner classes.

\begin{lem}\label{Lemma 2.2}
Assume that $E[p]\vert_{G_K}$ is irreducible.
Then $H^1(K[S],\Tc)$ is free over $\Lambda_\ac$.
\end{lem}

\begin{proof}
Let $M_S:=H^1(K[S],\Tc)$, and identify $\Lambda_\ac\simeq\bZ_p[[X]]$ via $\gamma_\ac\mapsto 1+X$.
We first claim that the two maps
\[
\alpha:M_S\overset{X}\longrightarrow M_S,\quad
\beta:M_S/XM_S\overset{p}\longrightarrow M_S/XM_S
\]
are injective. Indeed, the irreducibility assumption on $E[p]\vert_{G_K}$ implies that $T^{G_{K_\infty^\ac}}=\{0\}$,
and hence the injectivity of $\alpha$ follows from \cite[\S{1.3.3}]{PR:Lp}.
We thus get an injection $M_S/XM_S\hookrightarrow H^1(K[S],T)$, and so
to establish the injectivity of $\beta$ is suffices to show the injectivity of the map
\[
H^1(K[S],T)\overset{p}\longrightarrow H^1(K[S],T),
\]
but this follows again from the $G_K$-irreducibility of $E[p]$.
By the structure theorem for finitely generated modules over $\Lambda_\ac$,
the above shows that $M_S$ injects into a free module of finite rank with finite cokernel $N$. If $N\neq\{0\}$, then
${\rm Tor}^{\Lambda_\ac}_1(N,\Lambda_\ac/X\Lambda_\ac)$ is a nonzero $\bZ_p$-torsion module injecting into $M_S/XM_S$, contradicting the injectivity of $\beta$. Hence $N=\{0\}$ and $M_S$ is free over $\Lambda_\ac$.
\end{proof}

Set $\omega_n^\pm:=\omega_n^\pm((1+Y)^{p^a}-1)$ to lighten the
notation. A straightforward induction argument using 
$(\ref{(10)})$ shows that
\[
\omega_n^\epsilon z_n[S]=0,
\]
where $\eps$ is the sign $(-1)^n$ (see \cite[Lem.~4.2]{darmon-iovita}). By Lemma~\ref{lem:UN} 
Lemma~\ref{Lemma 2.2},
this implies that there is a unique class
\[
z_n[S]^\eps\in H^1(K_n^\ac[S],T)/\omega_n^\eps H^1(K^\ac_n[S],T)
\]
such that
\[
\tilde{\omega}_n^{-\eps}z_n[S]^\eps=
(-1)^{[\frac{n+1}{2}]}z_n[S].
\]

\begin{lem}\label{lem:DI-2.9}
For each $\eps\in\{\pm\}$
the sequences $\{z_n[S]^\eps\}_{n\equiv\eps\pmod{2}}$ are
compatible under the natural projections
\[
H^1(K_n^\ac[S],T)/\omega_n^\eps H^1(K_n^\ac[S],T)
\longrightarrow H^1(K_{n-2}^\ac[S],T)/\omega_{n-2}^\epsilon H^1(K_{n-2}^\ac[S],T)
\]
induced by corestriction.
\end{lem}

\begin{proof}
In light of the freeness result of Lemma~\ref{Lemma 2.2}, the argument in
\cite[Lem.~2.9]{darmon-iovita} applies verbatim.
\end{proof}

For every $\eps\in\{\pm\}$ and $S>0$ prime to $Np$ we may thus define classes $\mathbf{z}[S]^\eps\in H^1(K[S],\Tc)$ by
\begin{equation}\label{def:pm-HP}
\mathbf{z}[S]^\eps:=\varprojlim_n z_n[S]^\eps,
\end{equation}
where the limit is over $n$ with
the fixed parity determined by $\eps$. Since $\{\omega_n^\eps\}_n$
forms a basis for the topology of $\Lambda_\ac$, the class $\mathbf{z}[S]^\eps$ is well-defined.




\subsection{Explicit reciprocity law}\label{subsec:HP-ERL}

As we show in this section, similarly as in \cite{cas-hsieh1} for ordinary primes, the classes
$\mathbf{z}^\pm:=\mathbf{z}[1]^\pm$ satisfy an explicit reciprocity
law relating them to some of the anticyclotomic $p$-adic $L$-functions in $\S\ref{sec:anti-L}$.

Recall from $\S\ref{subsec:local-p}$ the element $d=\{d_m\}_m\in\varprojlim_m\cO_{k^m}^\times$
generating this $\bZ_p[[U]]$-module, and let
\[
F_{d,2}:=\varprojlim_m\sum_{\sigma\in U/p^mU}d_m^\sigma\cdot\sigma^2,
\]
viewed as an element in $\Lambda_{\unr}$. By the discussion in \cite[\S{6.4}]{LZ2}
on the Katz $p$-adic $L$-function (see also [\emph{loc.cit.}, \S{3.2}]), the quotient
$\mathscr{L}_{\pp}^{\tt Katz}(K)/F_{d,2}$ gives rise to a nonzero element in
$\Lambda_\ac$, rather than just $\Lambda_{\unr}^\ac$. 

\begin{lem}
We have the equality
\[
F_{d,2}=\biggl(\varprojlim_m\sum_{\sigma\in U/p^mU}d_m^\sigma\cdot\sigma\biggr)^2
\]
up to a unit in $\bZ_p[[U]]^\times$.
\end{lem}

\begin{proof}
This follows from a straightforward calculation.
\end{proof}

Thus setting $F_d:=\varprojlim_m\sum_{\sigma\in U/p^mU}d_m^\sigma\cdot\sigma$, the $p$-adic $L$-function $L_\pp(f/K)$ 
of Theorem~\ref{thm:wan} may be written as the product
\begin{equation}\label{eq:factor-2}
L_\pp(f/K)=\mathcal{L}_\pp(f/K)\cdot F_d^2\cdot U,
\end{equation}
for some $\mathcal{L}_\pp(f/K)\in\Lambda$ and $U\in\Lambda^\times$, and letting $\mathcal{L}_{\pp,\ac}(f/K)\in\Lambda_{\ac}$ be the image of $\mathcal{L}_\pp(f/K)$ under the projection $\Lambda\rightarrow\Lambda_\ac$, we see from Corollary~\ref{cor:wan-bdp}
that
\begin{equation}\label{eq:factor}
\mathscr{L}_{\pp}^{\tt BDP}(f/K)^2=\mathcal{L}_{\pp,\ac}(f/K)\cdot F_d^2\cdot U'
\end{equation}
for some $U'\in\Lambda_{\rm ac}^\times$. 

Note that for every $\eps\in\{\pm\}$ and $v\mid p$ the
classes $\mathbf{z}^\eps$ 
satisfy ${\rm loc}_v(\mathbf{z}^\eps)\in H^1_\eps(K,\Tc)$,
and so we may consider the image of ${\rm loc}_v(\mathbf{z}^\eps)$ under
the signed logarithm map ${\rm Log}_{\ac}^\eps$ constructed in $\S\ref{subsec:PR-maps}$.

\begin{thm}[Explicit reciprocity law]\label{3.1}
For every $\eps\in\{\pm\}$, we have the equality
\begin{equation}\label{eq:ERL-HP}
{\rm Log}_{\ac}^\eps({\rm loc}_\pp(\mathbf{z}^\eps))=\mathscr{L}_{\pp}^{\tt BDP}(f/K)\cdot F_d\cdot\sigma_{-1,\pp},
\end{equation}
where $\sigma_{-1,\pp}:={\rm rec}_\pp(-1)\vert_{K_{\infty}^\ac}\in\Gamma^\ac$.
In particular, the 
class ${\rm loc}_\pp(\mathbf{z}^\eps)$ is non-torsion.
\end{thm}

\begin{proof}
We give the proof for $\eps=+$, the proof for the other sign being virtually the same.
Let $\psi$ be an anticyclotomic Hecke character of infinity type $(1,-1)$ and conductor prime to $p$,
and let $\mathscr{L}_{\mathfrak{p},\psi}(f)\in R_0[[\Gamma^\ac]]$ be as in
\cite[Def.~3.5]{cas-hsieh1}. The $p$-adic $L$-function $\mathscr{L}^{\tt BDP}_\pp(f/K)$ of Theorem~\ref{thm:bdp} is then
given by
\[
\mathscr{L}^{\tt BDP}_{\mathfrak{p}}(f/K)={\rm Tw}_{\psi^{-1}}(\mathscr{L}_{\mathfrak{p},\psi}(f)),
\]
where ${\rm Tw}_{\psi^{-1}}:\unr[[\Gamma^\ac]]\rightarrow\unr[[\Gamma^\ac]]$ is the $\unr$-linear
isomorphism given by $\gamma\mapsto\psi^{-1}(\gamma)\gamma$ for $\gamma\in\Gamma^{\ac}$.
Let $\phi:\Gamma^\ac\rightarrow\mu_{p^\infty}$ be a nontrivial finite order character,
let $n>0$ be the smallest positive integer such that $\phi$ factors through $\Gamma^\ac/p^n\Gamma^\ac$ (using additive notation), and
assume that $n$ is even. Following the calculations in \cite[Thm.~4.8]{cas-hsieh1}, we then find that
\[
\begin{split}
\mathscr{L}^{\tt BDP}_{\mathfrak{p}}(f/K)(\phi^{-1})&=
\mathfrak{g}(\phi^{-1})\phi(p^n)p^{-n}
\sum_{\sigma\in\Gamma^\ac/p^n\Gamma^\ac}\phi(\sigma){\rm log}_{\hat{E}}(\sigma P[p^n])\\
&=\phi(-1)\cdot\frac{\phi(p^n)}{\mathfrak{g}(\phi)}\cdot(-1)^{n/2}\tilde{\omega}_n^-(\phi)
\sum_{\sigma\in\Gamma^\ac/p^n\Gamma^\ac}\phi(\sigma){\rm log}_{\hat{E}}(\sigma P[p^n]^+),
\end{split}
\]
where we used the definition of $P[p^n]^+$ for the second equality.
Combined with the interpolation properties of the map ${\rm Log}^+$ (see Lemma~\ref{lem:interpolation}),
this shows that
\begin{equation}\label{eq:end-calc}
\begin{split}
\mathscr{L}^{\tt BDP}_{\mathfrak{p}}(f/K)(\phi^{-1})&=
\phi(-1)\cdot\frac{\phi(p^n)}{\mathfrak{g}(\phi)}
\sum_{\sigma\in\Gamma^\ac/p^n\Gamma^\ac}\phi(\sigma)\log_{\hat{E}}(c_n^\sigma)
\cdot{\rm Log}_{\ac}^+({\rm loc}_\pp(\mathbf{z}_{}^+))(\phi^{-1})\\
&=\phi(-1)\sum_{\sigma\in\Gamma^\ac/p^n\Gamma^\ac}\phi(\sigma)d_{n+a}^\sigma
\cdot{\rm Log}_{\ac}^+({\rm loc}_\pp(\mathbf{z}_{}^+))(\phi^{-1}).
\end{split}
\end{equation}
Letting $\phi$ vary, equality (\ref{eq:ERL-HP}) follows immediately from $(\ref{eq:end-calc})$.

With the explicit reciprocity law $(\ref{eq:ERL-HP})$ in hand, the nontrivality of ${\rm loc}_\pp(\mathbf{z}^\pm)$ follows from the nonvanishing of $\mathscr{L}^{\tt BDP}_\pp(f/K)$ in Theorem~\ref{thm:bdp}.
\end{proof}

\begin{rem}
That the classes $\mathbf{z}^\pm$ are non-torsion over $\Lambda_\ac$ can also be deduced from the proof by
Cornut--Vatsal of Mazur's conjecture on higher Heegner points (see \cite[Thm.~1.5]{CV-dur} and the discussion right after it). 
However, the local refinement provided by Theorem~\ref{3.1}
will be a vital ingredient for our main results in this paper.
\end{rem}

\subsection{Anticyclotomic main conjectures}



By $(\ref{eq:factor-2})$ and $(\ref{eq:factor})$, there is an element
$\mathcal{L}_\pp^{\tt BDP}(f/K)$ in $\Lambda_\ac$ such that
\begin{equation}\label{eq:equivalences}
(\mathscr{L}^{\tt BDP}_\pp(f/K)^2)=(\mathcal{L}_\pp^{\tt BDP}(f/K)^2)=(\mathcal{L}_{\pp,\ac}(f/K))
\end{equation}
as principal ideals of $\unr[[\Gamma^\ac]]$. The Iwasawa--Greenberg main conjecture \cite{Greenberg55} for the 
$p$-adic $L$-function $\mathscr{L}^{\tt BDP}_\pp(f/K)$
of Theorem~\ref{thm:bdp} may thus be formulated as follows: 

\begin{conj}[Iwasawa--Greenberg main conjecture]\label{conj:IG}
The module $\mathfrak{X}^{{\rm rel},{\rm str}}(K,\Ac)$ is $\Lambda_\ac$-torsion, and
\[
{\rm Char}_{\Lambda_\ac}(\mathfrak{X}^{{\rm rel},{\rm str}}(K,\Ac))=(\mathcal{L}_\pp^{\tt BDP}(f/K)^2)
\]
as ideals in $\Lambda_{\ac}\otimes_{\bZ_p}\bQ_p$.
\end{conj}

As we show in Section~\ref{subsec:main-1}, Conjecture~\ref{conj:IG} is intimatelly related to the analogue of Perrin-Riou's Heegner point main conjecture \cite{PR-HP} formulated in the Introduction of this paper (see Conjecture~\ref{conj:PR-ss}). The following three lemmas will be used to relate the two, where we let $\eps\in\{\pm\}$ be a fixed sign.

\begin{lem}\label{thm:ES}
Assume that $\mathfrak{Sel}^{\eps,\eps}(K,\Tc)$ has $\Lambda_\ac$-rank $1$. Then
\begin{equation}\label{eq:pm-rel}
\mathfrak{Sel}^{\eps,\eps}(K,\Tc)=\mathfrak{Sel}^{\eps,{\rm rel}}(K,\Tc)
\end{equation}
and $\mathfrak{X}^{\eps,{\rm str}}(K,\Ac)$ is a torsion $\Lambda_\ac$-module.
\end{lem}

\begin{proof}
Consider the exact sequence
\begin{equation}\label{eq:es0}
\mathfrak{Sel}^{\eps,\eps}(K,\Tc)\overset{{\rm loc}_\pp}\longrightarrow H^1_{\eps}(K_\pp,\Tc)
\longrightarrow\mathfrak{X}^{{\rm rel},\eps}(K,\Ac)\longrightarrow \mathfrak{X}^{\eps,\eps}(K,\Ac)\longrightarrow 0.
\end{equation}
Since $\mathbf{z}^\eps$ lands in $\mathfrak{Sel}^{\eps,\eps}(K,\Tc)$, by Theorem~\ref{3.1} the image of the map ${\rm loc}_\pp$ is not $\Lambda_\ac$-torsion, and since $H^1_{\eps}(K_\pp,\Tc)$ has $\Lambda_{\rm ac}$-rank $1$, it follows
that ${\rm coker}({\rm loc}_\pp)$ is $\Lambda_{\rm ac}$-torsion. On the other hand, by Proposition~\ref{prop:str-rel} below the assumption in the lemma implies
that $\mathfrak{X}^{\eps,\eps}(K,\Ac)$ has $\Lambda_{\rm ac}$-rank $1$,
and so from (\ref{eq:es0})  we conclude that
\begin{equation}\label{eq:rank=1}
{\rm rank}_{\Lambda_{\rm ac}}(\mathfrak{X}^{{\rm rel},\eps}(K,\Ac))=1.
\end{equation}
Since $\mathfrak{X}^{{\rm rel},\eps}(K,\Ac)\simeq\mathfrak{X}^{\eps,{\rm rel}}(K,\Ac)$ by the action
of complex conjugation, we deduce from (\ref{eq:rank=1}) and Lemma~\ref{lem:str-rel}
that $\mathfrak{X}^{\eps,{\rm str}}(K,\Ac)$ is $\Lambda_{\rm ac}$-torsion.
Finally, since 
$H^1(K_{\overline{\pp}},\Tc)/H^1_{\eps}(K_{\overline{\pp}},\Tc)$ has $\Lambda_{\rm ac}$-rank $1$,
counting ranks in the exact sequence
\begin{equation}\label{eq:es1}
\begin{split}
0\longrightarrow\mathfrak{Sel}^{\eps,\eps}(K,\Tc)\longrightarrow
\mathfrak{Sel}^{\eps,{\rm rel}}(K,\Tc)
&\overset{{\rm loc}_{\overline\pp}}\longrightarrow
\frac{H^1(K_{\overline{\pp}},\Tc)}{H^1_{\eps}(K_{\overline{\pp}},\Tc)}\\
&\longrightarrow\mathfrak{X}^{\eps,\eps}(K,\Ac)\longrightarrow \mathfrak{X}^{\eps,{\rm str}}(K,\Ac)\longrightarrow 0,\nonumber
\end{split}
\end{equation}
we see that $\mathfrak{Sel}^{\eps,\eps}(K,\Tc)$ and $\mathfrak{Sel}^{\eps,{\rm rel}}(K,\Tc)$ have both
$\Lambda_{\rm ac}$-rank $1$,
and since the quotient  $H^1(K_{\overline\pp},\Tc)/H^1_{\eps}(K_{\overline\pp},\Tc)$ is also
$\Lambda_{\rm ac}$-torsion-free,  
equality (\ref{eq:pm-rel}) follows.
\end{proof}


\begin{lem}\label{lem:cas}
Assume that $\mathfrak{Sel}^{\eps,\eps}(K,\Tc)$ has $\Lambda_\ac$-rank $1$. Then
for any height one prime $\mathfrak{P}$ of $\Lambda_{\ac}$ we have
\[
{\rm ord}_{\mathfrak{P}}(\mathcal{L}^{\tt BDP}_{\pp}(f/K))={\rm length}_{\mathfrak{P}}({\rm coker}({\rm loc}_{{\pp}}))
+{\rm length}_{\mathfrak{P}}\biggl(\frac{\mathfrak{Sel}^{\eps,\eps}(K,\Tc)}{\Lambda^\ac\cdot\mathbf{z}_{}^\eps}\biggr),
\]
where ${\rm loc}_{\pp}:\mathfrak{Sel}^{\eps,\eps}(K,\Tc)
\rightarrow H^1_\eps(K_{\pp},\Tc)$ is the natural restriction map.
\end{lem}

\begin{proof}
Consider the tautological exact sequence
\begin{equation}\label{eq:tauto-es}
0\longrightarrow\mathfrak{Sel}^{{\rm str},\eps}(K,\Tc)\longrightarrow
\mathfrak{Sel}^{\eps,\eps}(K,\Tc)
\longrightarrow H^1_{\eps}(K_\pp,\Tc)
\longrightarrow{\rm coker}({\rm loc}_\pp)\longrightarrow 0.
\end{equation}
By Theorem~\ref{3.1}, the image of $\mathbf{z}_{}^\eps\in\mathfrak{Sel}^{\eps,\eps}(K,\Tc)$
under the map ${\rm loc}_\pp$ is not $\Lambda_\ac$-torsion.
Since $\mathfrak{Sel}^{\eps,\eps}(K,\Tc)$ has $\Lambda_\ac$-rank $1$ by assumption,  
this shows that $\mathfrak{Sel}^{{\rm str},\eps}(K,\Tc)$
is $\Lambda_\ac$-torsion, and since $H^1(K,\Tc)$ is $\Lambda_\ac$-torsion-free
(see e.g. \cite[Lem.~2.2.9]{howard-PhD-I}), it follows that
\begin{equation}\label{eq:str=0}
\mathfrak{Sel}^{{\rm str},\eps}(K,\Tc)=\{0\}.
\end{equation}

From $(\ref{eq:tauto-es})$ we thus deduce the exact sequence
\begin{equation}\label{eq:1}
0\longrightarrow\frac{\mathfrak{Sel}^{\eps,\eps}(K,\Tc)}
{\Lambda_\ac\cdot\mathbf{z}_{}^\eps}
\longrightarrow\frac{H^1_\eps(K_{\pp},\Tc)}
{\Lambda_\ac\cdot{\rm loc}_{\pp}(\mathbf{z}_{}^\eps)}
\longrightarrow {\rm coker}({\rm loc}_{\pp})\longrightarrow 0,\nonumber
\end{equation}
and since by the explicit reciprocity law of Theorem~\ref{3.1}
the map ${\rm Log}_{\ac}^\eps$ induces a $\Lambda_\ac$-module isomorphism
\[
\frac{H^1_{\eps}(K_\pp,\Tc)}{\Lambda_\ac\cdot{\rm loc}_\pp(\mathbf{z}^\eps)}
\overset{\simeq}\longrightarrow\frac{\Lambda_\ac}{\Lambda_\ac\cdot\mathcal{L}^{\tt BDP}_\pp(f/K)},
\]
the result follows.
\end{proof}

\begin{lem}\label{lem:tors}
Assume that $\mathfrak{Sel}^{\eps,\eps}(K,\Tc)$ has $\Lambda_\ac$ rank $1$.
Then the module $\mathfrak{X}^{{\rm rel},{\rm str}}(K,\Ac)$ is $\Lambda_\ac$-torsion, and
for any height one prime $\mathfrak{P}$ of $\Lambda_\ac$ we have
\[
{\rm length}_{\mathfrak{P}}(\mathfrak{X}^{{\rm rel},{\rm str}}(K,\Ac))=
{\rm length}_{\mathfrak{P}}(\mathfrak{X}^{\eps,\eps}(K,\Ac)_{\rm tors})+
2\;{\rm length}_{\mathfrak{P}}({\rm coker}({\rm loc}_{\pp})),
\]
where ${\rm loc}_{\pp}:\mathfrak{Sel}^{\eps,\eps}(K,\Tc)
\rightarrow H^1_\eps(K_{\pp},\Tc)$ is the natural restriction map.
\end{lem}

\begin{proof}
Global duality yields the exact sequence
\begin{equation}\label{PT3}
0\longrightarrow{\rm coker}({\rm loc}_{\pp})\longrightarrow
\mathfrak{X}^{{\rm rel},\eps}(K,\Ac)\longrightarrow
\mathfrak{X}^{\eps,\eps}(K,\Ac)\longrightarrow 0.
\end{equation}
As shown in the proof of
Lemma~\ref{lem:cas}, the first term in the sequence is $\Lambda_\ac$-torsion;
since by Proposition~\ref{prop:str-rel} below the assumption implies that $\mathfrak{X}^{\eps,\eps}(K,\Ac)$
has $\Lambda_\ac$-rank $1$, this shows that the same is true for $\mathfrak{X}^{{\rm rel},\eps}(K,\Ac)$,
and by Lemma~\ref{lem:str-rel} it follows that
$\mathfrak{X}^{{\rm str},\eps}(K,\Ac)$ is $\Lambda_\ac$-torsion.
Thus taking $\Lambda_\ac$-torsion in $(\ref{PT3})$ and using Lemma~\ref{lem:str-rel} again,
it follows that
\begin{equation}\label{pbar}
{\rm length}_{\mathfrak{P}}(\mathfrak{X}^{{\rm str},\eps}(K,\Ac))=
{\rm length}_{\mathfrak{P}}(\mathfrak{X}^{\eps,\eps}(K,\Ac)_{{\rm tors}})+
{\rm length}_{\mathfrak{P}}({\rm coker}({\rm loc}_{\pp}))
\end{equation}
for any height one prime $\mathfrak{P}$ of $\Lambda_\ac$.

Another application of global duality yields the exact sequence
\begin{equation}\label{PT4}
0\longrightarrow{\rm coker}({\rm loc}_{\pp}^{\rm rel})\longrightarrow
\mathfrak{X}^{{\rm rel},{\rm str}}(K,\Ac)\longrightarrow
\mathfrak{X}^{\eps,{\rm str}}(K,\Ac)\longrightarrow 0,
\end{equation}
where ${\rm loc}^{\rm rel}_{\pp}:\mathfrak{Sel}^{\eps,{\rm rel}}(K,\Tc)\rightarrow
H^1_\eps(K_\pp,\Tc)$ is the natural restriction map. By Lemma~\ref{thm:ES},
this is the same as the map ${\rm loc}_\pp$ in the statement, and hence
${\rm coker}({\rm loc}^{\rm rel}_{{\pp}})={\rm coker}({\rm loc}_\pp)$ is $\Lambda_\ac$-torsion.
Since $\mathfrak{X}^{\eps,{\rm str}}(K,\Ac)$ is $\Lambda_\ac$-torsion by Lemma~\ref{thm:ES},
we conclude from $(\ref{PT4})$ that $\mathfrak{X}^{{\rm rel},{\rm str}}(K,\Ac)$ is $\Lambda_\ac$-torsion.
Combining $(\ref{PT4})$ and $(\ref{pbar})$, we thus have
\begin{align*}
{\rm length}_{\mathfrak{P}}(\mathfrak{X}^{{\rm rel},{\rm str}}(K,\Ac))
&= {\rm length}_{\mathfrak{P}}(\mathfrak{X}^{\eps,{\rm str}}(K,\Ac))+
{\rm length}_{\mathfrak{P}}({\rm coker}({\rm loc}_{{\pp}})) \\
&={\rm length}_{\mathfrak{P}}(\mathfrak{X}^{\eps,\eps}(K,\Ac)_{\rm tors})+
2\;{\rm length}_{\mathfrak{P}}({\rm coker}({\rm loc}_{\pp}))
\end{align*}
for any height one prime $\mathfrak{P}$ of $\Lambda_\ac$,
as was to be shown.
\end{proof}

\subsection{Kolyvagin system argument}\label{subsec:KS}

The purpose of this section is to prove the following result, establishing under mild assumptions
one of the divisibility predicted by Conjecture~\ref{conj:PR-ss}.

\begin{thm}\label{thm:KS-argument}
Assume that ${\rm Gal}(\overline{\bQ}/K)\rightarrow{\rm Aut}_{\bZ_p}(T)$ is surjective, and that $E[p]$ is ramified at every prime $\ell\mid N^-$. Let $\eps\in\{\pm\}$. Then $\mathfrak{X}^{\eps,\eps}(K,\Ac)$ and $\mathfrak{Sel}^{\eps,\eps}(K,\Tc)$
both have $\Lambda_\ac$-rank $1$, and we have the divisibility
\[
{\rm Char}_{\Lambda_\ac}(\mathfrak{X}^{\eps,\eps}(K,\Ac)_{\rm tors})
\supseteq {\rm Char}_{\Lambda_\ac}\biggl(\frac{\mathfrak{Sel}^{\eps,\eps}(K,\Tc)}
{\Lambda_\ac\cdot\mathbf{z}^\eps}\biggr)^2.
\]
\end{thm}

For the proof of Theorem~\ref{thm:KS-argument}, we will adapt to our supersingular
setting the Kolyvagin system techniques developed by Howard \cite{howard-PhD-I} in the
ordinary case. More precisely, we will use the classes $P[S]^\eps$ introduced in (\ref{def:pm-HP})
to build a certain Kolyvagin system for $\mathbf{T}^\ac$; the nontriviality of this system
will follow from the nontriviality of $\mathbf{z}^\eps$ established in Theorem~\ref{3.1},
and Theorem~\ref{thm:KS-argument} will then follow from a suitable adaptation of
Howard's arguments.

As in \cite[\S{1.1}]{howard-PhD-I}, by a Selmer structure on $\Tc$ we mean
a choice of a local condition $H^1_{\mathcal{F}}(K_v,\Tc)\subseteq H^1(K_v,\Tc)$ for each place $v\in\Sigma$.
(Here 
$\Sigma$ is any finite set of places of $K$ containing those above $p$, those above $\infty$, and those
where $V$ ramifies.) We define the Selmer structure $\mathcal{F}^\pm$ on $\Tc$ to be the unramified local condition at
the places in $\Sigma$ not dividing $p$, and the plus/minus local condition $H^1_\pm(K_v,\Tc)$ at the primes $v$ above $p$.

For the statement of the next result, we refer the reader to \cite[\S{1.2}]{howard-PhD-I}
for the definition of the module of Kolyvagin systems $\mathbf{KS}(\Tc,\mathcal{F},\mathcal{L})$
attached to a Selmer structure $\mathcal{F}$ on $\Tc$ and a certain set $\mathcal{L}$
of primes inert in $K$.

\begin{thm}\label{thm:pm-HP-KS}
For each $\eps\in\{\pm\}$ there exists a Kolyvagin system $\kappa^{\eps}\in\mathbf{KS}(\Tc,\mathcal{F}^\eps,\mathcal{L})$
with $\kappa_1^\eps=\mathbf{z}^\eps$.
\end{thm}

\begin{proof}
Let $\mathcal{L}_0$ be the set of rational primes $\ell$ not dividing $pN$ and inert in $K$.
For each $\ell\in\mathcal{L}_0$, let $\lambda$ be the prime of $K$ above $\ell$, and denote
by $I_\ell$ the smallest ideal of $\Lambda^\ac$ containing $\ell+1$ for which the Frobenius element
${\rm Fr}_\lambda\in G_{K_\lambda}$ acts trivially on $\Tc/I_\ell\Tc$. Let $\mathcal{L}=\mathcal{L}(\Tc)\subseteq\mathcal{L}_0$
consist of the primes $\ell\in\mathcal{L}_0$ with $I_\ell\subseteq p\bZ_p$, and let $\mathcal{N}$ be the set
of square-free products $S$ of primes in $\mathcal{L}$, with the convention that $1\in\mathcal{N}$.
For each $S\in\mathcal{N}$, define $I_S=\sum_{\ell\vert S}I_\ell$.

Applied to the Heegner classes $\mathbf{z}[S]^\eps$ defined in $(\ref{def:pm-HP})$,
the derivative construction in \cite[\S{1.7}]{howard-PhD-I} produces classes
\[
\kappa_S^\eps\in H^1(K,\mathbf{T}^\ac/I_S\mathbf{T}^\ac),
\]
indexed by the products $S\in\mathcal{N}$, with $\kappa_1^\eps=\mathbf{z}[1]^\eps=\mathbf{z}^\eps$.
The verification that these classes form a Kolyvagin system for the Selmer structure $\mathcal{F}^\eps$ on
$\mathbf{T}^\ac$ follows from the same argument as in \cite[Lem.~2.3.4]{howard-PhD-I} (as extended in \cite[Prol.~3.4.1]{howard-PhD-II} to cover the primes $v\mid N^-$), 
the only
difference being at the primes $v\mid p$, where we are led to show that the localization of
$\kappa_S^\eps$ at $v$ is contained in $H^1_{\mathcal{F}^\eps}(K_v,\Tc/I_S\Tc)$, defined as the
image of the natural map
\[
H^1_{\mathcal{F}^\eps}(K_v,\Tc)\longrightarrow H^1(K_v,\Tc/I_S\Tc).
\]
But this follows from the same argument as in the proof of Lemma~\ref{Lemma 2.2}.
\end{proof}

Let $\mathfrak{P}$ be a height one prime of $\Lambda_\ac$, let $S_\mathfrak{P}$
denote the integral closure of $\Lambda_\ac/\mathfrak{P}$, and let
$\varpi_\mathfrak{P}\in S_{\mathfrak{P}}$ be a uniformizer. Define the Galois representations
\[
T_{\mathfrak{P}}:=\Tc\otimes_{\Lambda_\ac}S_{\mathfrak{P}},\quad\quad
A_{\mathfrak{P}}:=\Ac\otimes_{\Lambda_\ac}S_{\mathfrak{P}},
\]
and let $T_{\mathfrak{P},m}$ and $A_{\mathfrak{P},m}$ be their reduction modulo $\varpi_{\mathfrak{P}}^m$.

For each place $v\mid p$ in $K$,
define $H^1_\pm(K_v,T_{\mathfrak{P}})\subseteq H^1(K_v,T_{\mathfrak{P}})$ to be the
image of the natural map
\[
H^1_{\mathcal{F}^\pm}(K_v,\Tc)\longrightarrow H^1(K_v,\Tc\otimes_{\Lambda^\ac}\Lambda^\ac/\mathfrak{P})
\longrightarrow H^1(K_v,T_\mathfrak{P}),
\]
and define $H^1_\pm(K_v,T_{\mathfrak{P},m})\subseteq H^1(K_v,T_{\mathfrak{P},m})$ in the same manner.
Also, let $H^1_\pm(K_v,A_{\mathfrak{P}})$ (resp. $H^1_\pm(K_v,A_{\mathfrak{P},m})$)
be the orthogonal complement of $H^1_\pm(K_v,T_{\mathfrak{P}})$ (resp. $H^1_\pm(K_v,T_{\mathfrak{P},m})$)
under the local Tate pairing.

\begin{lem}\label{eq:2.2.7}
For every height one prime $\mathfrak{P}\subseteq\Lambda_\ac$ with $\mathfrak{P}\neq p\Lambda_\ac$,
and for every place $v$ of $K$, the natural maps
\begin{align*}
H^1_{\mathcal{F}^\pm}(K_v,\Tc/\mathfrak{P}\Tc)&\longrightarrow
H^1_{\mathcal{F}^\pm_{\mathfrak{P}}}(K_v,T_{\mathfrak{P}}),\quad\\
H^1_{\mathcal{F}^\pm_{\mathfrak{P}}}(K_v,A_{\mathfrak{P}})&\longrightarrow
H^1_{\mathcal{F}^\pm}(K_v,\Ac[\mathfrak{P}])
\end{align*}
have finite kernel and cokernel of order bounded by a constant depending only on
$[S_\mathfrak{P}:\Lambda_\ac/\mathfrak{P}]$.
\end{lem}

\begin{proof}
Let $m$ be any positive integer, and $n\gg 0$ be such that we have the inclusion of
ideals of $\Lambda_\ac$:
\[
(\omega_n(X),p^m)\subseteq (\mathfrak{P}, p^m).
\]
Then it follows from \cite[Prop.~3.14]{kim-parity} (\emph{cf.} \cite[Prop.~4.11]{kim-parity})
that for each $v\mid p$ in $K$, the module $H^1_{\pm}(K_v,T_{\mathfrak{P},m})$ is
the exact annihilator of $H^1_{\pm}(K_{\overline{v}},T_{\mathfrak{P},m})$ under
local Tate duality. In particular, $H^1_\pm(K_v,A_{\mathfrak{P},m})$ can be identified with
$H^1_{\pm}(K_{\overline{v}},T_{\mathfrak{P},m})$. Hence by \cite[Prop.~4.18]{kim-parity}
the second map in the statement has kernel and cokernel with the required bounds,
and taking duals the same properties for the first map follow.
\end{proof}

\begin{prop}\label{prop:str-rel}
For every $\eps\in\{\pm\}$ we have 
\[
{\rm rank}_{\Lambda_\ac}(\mathfrak{Sel}^{\eps,\eps}(K,\Tc))
={\rm rank}_{\Lambda_\ac}(\mathfrak{X}^{\eps,\eps}(K,\Ac)).	
\]
\end{prop}

\begin{proof}
Since the Selmer structure $\mathcal{F}^\eps$ introduced above is such that
\[
\mathfrak{Sel}^{\eps,\eps}(K,\Tc)=H^1_{\mathcal{F}^\eps}(K,\Tc),
\] 
the result follows from Lemma~\ref{eq:2.2.7} in the same manner that Proposition~2.2.8 in \cite{howard-PhD-I} is deduced from \cite[Lem.~2.2.7]{howard-PhD-I} (see also \cite[Lem.~3.5]{wan}).
\end{proof}

\begin{proof}[Proof of Theorem~\ref{thm:KS-argument}]
We can now adapt the argument in the proof \cite[Thm.~2.2.10]{howard-PhD-I}.
Indeed, for every height one prime $\mathfrak{P}\subseteq\Lambda_\ac$ with $\mathfrak{P}\neq p\Lambda_\ac$,
we have a map
\[
\mathbf{KS}(\mathbf{T}^\ac,\mathcal{F}^\eps,\mathcal{L}(\Tc))\longrightarrow
\mathbf{KS}(T_{\mathfrak{P}},\mathcal{F}^\eps_\mathfrak{P},\mathcal{L}(T_{\mathfrak{P}})),
\]
where $\mathcal{F}^\eps_{\mathfrak{P}}$ is the Selmer structure on $T_{\mathfrak{P}}$
naturally induced from $\mathcal{F}^\eps$. Letting $\kappa^\eps(\mathfrak{P})$ be the image of the
Kolyvagin system $\kappa^\eps$ of Theorem~\ref{thm:pm-HP-KS} under this map, it follows
from Theorem~\ref{3.1} and Lemma~\ref{eq:2.2.7} that $\kappa_1^\eps(\mathfrak{P})$
generates and infinite $S_\mathfrak{P}$-submodule of $H^1_{\mathcal{F}_{\mathfrak{P}}^\eps}(K,T_{\mathfrak{P}})$
for all but finitely many $\mathfrak{P}$. To deduce, as in \cite[Prop.~2.1.3]{howard-PhD-I}, that
for any such $\mathfrak{P}$ the module $H^1_{\mathcal{F}^\eps_{\mathfrak{P}}}(K,T_{\mathfrak{P}})$
is free of rank one over $S_\mathfrak{P}$, it suffices to show that the triple
$(T_{\mathfrak{P}},\mathcal{F}_{\mathfrak{P}}^\eps,\mathcal{L}(T_{\mathfrak{P}}))$
satisfies the hypotheses (H.0)--(H.5) of [\emph{loc.cit.},\S{1.2}].
The only difference here with respect to the verification of these hypotheses in \cite[Prop.~2.1.3]{howard-PhD-I}
is the self-duality condition in hypothesis (H.4), but this follows from \cite[Prop.~3.14]{kim-parity}
as indicated above.

By Lemma~\ref{eq:2.2.7}, this shows that $H^1_{\mathcal{F}^\eps}(K,\Tc)\otimes_{\Lambda_\ac}S_\mathfrak{P}$
is a free $S_{\mathfrak{P}}$-module of rank $1$, from where the first part of Theorem~\ref{thm:KS-argument}
follows immediately, and for the second part the argument in \cite[Thm.~2.2.10]{howard-PhD-I} applies verbatim.
\end{proof} 

\section{Main results}


\subsection{Proof of the main conjectures}\label{subsec:main-1}

For the ease of notation, set
\[
\mathfrak{X}_\pp(K,\mathbf{A}):=\mathfrak{X}^{{\rm rel},{\rm str}}(K,\mathbf{A}),
\]
and similarly for $\mathfrak{X}_\pp(K,\Ac)$.



\begin{thm}\label{thm:howard}
Let $E/\bQ$ be an elliptic curve of conductor $N$ with good supersingular reduction at $p$,
and let $K/\bQ$ be an imaginary quadratic field of discriminant prime to $N$. Assume that
the triple $(E,p,K)$ satisfy the following:
\begin{itemize}
\item{} $p\geqslant 5$,
\item{} $p=\pp\overline{\pp}$ splits in $K$,
\item{} hypothesis {\rm (\ref{eq:gen-HH})} holds,
\item{} $N$ is square-free,
\item{} $N^-\neq 1$,
\item{} $E[p]$ is ramified at every prime $\ell\vert N^-$,
\item{} ${\rm Gal}(\overline{\bQ}/K)\rightarrow{\rm Aut}_{\bZ_p}(T_p(E))$ is surjective.
\end{itemize}
Then:
\begin{enumerate}
\item{} 
For each $\eps\in\{\pm\}$ both $\mathfrak{X}^{\eps,\eps}(K,\Ac)$ and $\mathfrak{Sel}^{\eps,\eps}(K,\Tc)$
have $\Lambda_\ac$-rank $1$, and
\[
{\rm Char}_{\Lambda_\ac}(\mathfrak{X}^{\eps,\eps}(K,\Ac)_{\rm tors})
={\rm Char}_{\Lambda_\ac}\biggl(\frac{\mathfrak{Sel}^{\eps,\eps}(K,\Tc)}
{\Lambda_\ac\cdot\mathbf{z}^\eps}\biggr)^2
\]
as ideals in $\Lambda_\ac$.
\item{} 
$\mathfrak{X}_\pp(K,\Ac)$ is $\Lambda_\ac$-torsion, and
\[
{\rm Char}_{\Lambda_\ac}(\mathfrak{X}_\pp(K,\Ac))=(\mathcal{L}^{\tt BDP}_\pp(f/K)^2)
\]
as ideals in $\Lambda_\ac$.
\end{enumerate}
\end{thm}

\begin{proof}
By Theorem~\ref{thm:KS-argument} we know that $\mathfrak{Sel}^{\eps,\eps}(K,\Tc)$ has $\Lambda_\ac$-rank $1$.
By Proposition~\ref{prop:str-rel} the same is true for $\mathfrak{X}^{\eps,\eps}(K,\Ac)$, and by Lemma~\ref{lem:tors}
the module $\mathfrak{X}_\pp(K,\Ac)$ is $\Lambda_\ac$-torsion.
Let $\mathfrak{P}$ be a height one prime of $\Lambda_\ac$. Then by the divisibility in
Theorem~\ref{thm:KS-argument} we have
\begin{equation}\label{eq:upper}
{\rm length}_{\mathfrak{P}}(\mathfrak{X}^{\eps,\eps}(K,\Ac)_{\rm tors})
\leqslant 2\;{\rm length}_{\mathfrak{P}}
\biggl(\frac{\mathfrak{Sel}^{\eps,\eps}(K,\Ac)}{\Lambda_\ac\cdot\mathbf{z}_{}^\eps}\biggr).
\end{equation}
Combined with Lemma~\ref{lem:tors} and Lemma~\ref{lem:cas}, respectively, this implies that
\begin{align*}
{\rm length}_{\mathfrak{P}}(\mathfrak{X}_\pp(K,\Ac))
&\leqslant 2\;{\rm length}_{\mathfrak{P}}
\biggl(\frac{\mathfrak{Sel}^{\eps,\eps}(K,\Tc)}{\Lambda^\ac\cdot\mathbf{z}_{}^\eps}\biggr)+
2\;{\rm length}_{\mathfrak{P}}({\rm coker}({\rm loc}_{\pp}))\\
&=2\;{\rm ord}_{\mathfrak{P}}(\mathcal{L}^{\tt BDP}_{\pp}(f/K)),
\end{align*}
and hence we have the divisibility
\begin{equation}\label{eq:div}
{\rm Char}_{\Lambda_\ac}(\mathfrak{X}_\pp(K,\Ac))\;\supseteq\;
(\mathcal{L}_{\pp}^{\tt BDP}(f/K)^2).
\end{equation}

It remains to show the two divisibilities
$\subseteq$ in the theorem. Let $I^{\rm cyc}:=(\gamma_{\rm cyc}-1)\subset\Lambda$ be the principal ideal generated by $\gamma_{\rm cyc}-1$.
Similarly as in \cite[Prop.~3.9]{SU},
the natural restriction map $H^1(K_\infty^\ac,E[p^\infty])\rightarrow H^1(K_\infty,E[p^\infty])$
induces a $\Lambda_{\ac}$-module isomorphism
\begin{equation}\label{eq:control}
\mathfrak{X}_\pp(K,\mathbf{A})/I^{\rm cyc}\mathfrak{X}_\pp(K,\mathbf{A})
\simeq\mathfrak{X}_\pp(K,\Ac).
\end{equation}

By $(\ref{eq:factor})$,
the two-variable divisibility in \cite[Thm.~6.13]{wan-combined} thus yields the divisibility
\begin{equation}\label{eq:div-wan}
{\rm Char}_{\Lambda_\ac}(\mathfrak{X}_{\pp}(K,\Ac))\subseteq(\mathcal{L}^{\tt BDP}_\pp(f/K)^2)
\end{equation}
in $\Lambda_\ac\otimes_{\bZ_p}\bQ_p$; however, since $\mu(\mathcal{L}^{\tt BDP}_\pp(f/K))=0$ by Theorem~\ref{thm:mu-bdp}, the divisibility $(\ref{eq:div-wan})$ holds already in $\Lambda_\ac$, and therefore equality holds in  $(\ref{eq:div})$. This shows part (1) of Theorem~\ref{thm:howard}, and part (2) follows from it by Lemma~\ref{lem:cas} and Lemma~\ref{lem:tors}.
\end{proof}

\begin{thm}\label{thm:two-varIMC}
Let $\eps\in\{\pm\}$.
Under the hypotheses of Theorem~\ref{thm:howard}, the following hold:
\begin{enumerate}
\item{} $\mathfrak{X}^{\eps,{\rm str}}(K,\mathbf{A})$ is $\Lambda$-torsion,
$\mathfrak{Sel}^{\eps,{\rm rel}}(K,\mathbf{T})$ has $\Lambda$-rank $1$, and
\[
{\rm Char}_{\Lambda}(\mathfrak{X}^{\eps,{\rm str}}(K,\mathbf{A}))\cdot\mathcal{H}^\eps=
{\rm Char}_\Lambda\bigg(\frac{\mathfrak{Sel}^{\eps,{\rm rel}}(K,\mathbf{T})}{\Lambda\cdot\mathcal{BF}^\eps}\biggr)
\]
as ideals in $\Lambda$.
\item{} $\mathfrak{X}_\pp(K,\mathbf{A})$ is $\Lambda$-torsion, and
\[
{\rm Char}_{\Lambda}(\mathfrak{X}_\pp(K,\mathbf{A}))=(\mathcal{L}_\pp(f/K))
\]
as ideals in $\Lambda$.
\end{enumerate}
\end{thm}

\begin{proof}
Given the equalities in the anticyclotomic main conjecture established in Theorem~\ref{thm:howard},
we shall first deduce part (2) of the theorem by an anticyclotomic analogue of the argument in \cite[Thm.~3.30]{SU}.
Let $I^{\rm cyc}$ be the ideal of $\Lambda$ generated by $\gamma_{\rm cyc}-1$, let
\[
X:={\rm Char}_{\Lambda}(\mathfrak{X}_\pp(K,\mathbf{A})),\quad Y:=(\mathcal{L}_\pp(f/K)).
\]
Similarly as in
the proof of Theorem~\ref{thm:howard}, the divisibility in \cite[Thm.~6.13]{wan-combined}
yields the divisibility $X\subseteq Y$ as ideals in $\Lambda$. On the other hand,
in light of $(\ref{eq:control})$ and Corollary~\ref{cor:wan-bdp}, Theorem~\ref{thm:howard}
implies that $\mathfrak{X}_\pp(K,\mathbf{A})$ is $\Lambda$-torsion and that
$X=Y\pmod{I^{\rm cyc}}$. The equality $X=Y$ as ideals in $\Lambda$, i.e., the equality in part (2) of the theorem, thus follows from \cite[Lem.~3.2]{SU}; by Theorem~\ref{thm:2-varIMC}, the equality in part (1) also follows. 
\end{proof}

\begin{cor}\label{cor:kimIMC}
Let $\eps\in\{\pm\}$. Under the hypotheses of Theorem~\ref{thm:howard}, $\mathfrak{X}^{\eps,\eps}(K,\mathbf{A})$ is $\Lambda$-torsion, and
\[
{\rm Char}_{\Lambda}(\mathfrak{X}^{\eps,\eps}(K,\mathbf{A}))=(L_p^{\eps,\eps}(f/K))
\]
as ideals in $\Lambda$.
\end{cor}

\begin{proof}
In light of Theorem~\ref{thm:2-varIMC}, the result of Theorem~\ref{thm:two-varIMC}
implies that $\mathfrak{X}^{\eps,\eps}(K,\mathbf{A})$ is $\Lambda$-torsion, and we have the divisibility
\begin{equation}\label{eq:2var-div}
{\rm Char}_{\Lambda}(\mathfrak{X}^{\eps,\eps}(K,\mathbf{A}))\subseteq(L_p^{\eps,\eps}(f/K))
\end{equation}
as ideals in $\Lambda$. We note that \emph{a priori} this divisibility holds just in $\Lambda[1/P]$, 
but by the 
nonvanishing of the cyclotomic specialization of $L_p^{\eps,\eps}(f/K)$ (which follows from the nonvanishing
of the $p$-adic $L$-functions constructed by Kobayashi \cite{kobayashi-152} and the discussion in the paragraphs below),
the divisibility holds as stated.

Similarly as in the proof of Theorem~\ref{thm:two-varIMC}, we will deduce 
that equality holds in $(\ref{eq:2var-div})$ by an appropriate application of \cite[Lem.~3.2]{SU}.
Set
\[
X:={\rm Char}_\Lambda(\mathfrak{X}^{\eps,\eps}(K,\mathbf{A})),\quad\quad Y:=(L_p^{\eps,\eps}(f/K)),
\]
and let $I^\ac$ be the kernel of the canonical projection $\Lambda\twoheadrightarrow\Lambda_{\rm cyc}$. Let $\mathcal{L}_p^\pm(E/\bQ)$ and $\mathcal{L}^\pm_p(E^{(K)}/\bQ)$ be the $p$-adic $L$-functions constructed in \cite[\S{3}]{kobayashi-152} for the elliptic curve $E$ and its quadratic twist $E^{(K)}$, respectively, and set
\[
\mathcal{L}_p^{\eps,\eps}(f/K):=\mathcal{L}_p^\eps(E/\bQ)\cdot\mathcal{L}^\eps_p(E^{(K)}/\bQ).
\]

By the control theorem of \cite[Prop.~8.7]{wan-combined} and (see also \cite[Lem.~3.6]{SU}), the
divisibility in \cite[Thm.~4.1]{kobayashi-152} applied to $E$ and  $E^{(K)}$ yields the divisibility
\begin{equation}\label{eq:kob-div}
(X\;{\rm mod}\;I^{\ac})\supseteq (\mathcal{L}_p^{\eps,\eps}(f/K))
\end{equation}
as ideals in $\Lambda_{\rm cyc}$. Thus it remains
to compare the periods $\Omega_E^+\cdot\Omega_E^-$ used in the construction of $\mathcal{L}_p^{\eps,\eps}(f/K)$
with Hida's canonical period used in the construction of $L_p^{\eps,\eps}(f/K)$ in
Theorem~\ref{prop:loeffler+-}. By \cite[Lem.~9.5]{skinner-zhang}, the ratio of periods is a $p$-adic unit which we can clearly ignore. This shows that $(\ref{eq:2var-div})$ and $(\ref{eq:kob-div})$ yield the equality $X=Y\pmod{I^\ac}$, and so the equality
in Corollary~\ref{cor:kimIMC} follows from the divisibility $(\ref{eq:2var-div})$ by virtue of \cite[Lem.~3.2]{SU}.  
(Meanwhile we used the fact that by the argument in \cite[Lem.~8.6]{wan-combined} the two variable $p$-adic $L$-functions $L_p^{\eps,\eps}(f/K)$ are integral.)
\end{proof}

\subsection{A converse to Gross--Zagier--Kolyvagin}

Let
\[
y_K:=P_0[1]={\rm tr}^{K[1]}_K(P[1])\in E(K)
\]
be the Heegner point introduced in $\S$\ref{subsec:HP}.

Our next result is an analogue for supersingular primes
of Skinner's converse to a theorem of Gross--Zagier and Kolyvagin
(\emph{cf.} \cite[Thm.~B]{skinner}) holding under slightly weaker conditions than in \emph{loc.cit.}
(as explained in the Introduction before the statement of Theorem~B).

\begin{thm}
Let the hypotheses be as in Theorem~\ref{thm:howard}, and assume that
${\rm Sel}_{p^\infty}(f/K)$
has $\bZ_p$-corank $1$. Then $y_K\neq 0\in E(K)\otimes_{\bZ}\bQ$.
In particular, ${\rm ord}_{s=1}L(E/K,s)=1$.
\end{thm}

\begin{proof}
Let $\gamma_\ac\in\Gamma^{\ac}$ be a topological generator, and set
$I^{\ac}=(\gamma_{\ac}-1)\subseteq\Lambda_{\ac}$. We shall work with the sign $\eps=+1$. By Kobayashi's control theorem
(as extended by B.-D. Kim \cite{kim-CJM} to more general $\bZ_p$-extensions),
there is natural surjective map
\[
\mathfrak{X}^{+,+}(K,\mathbf{A}^\ac)/I^\ac \mathfrak{X}^{+,+}(K,\mathbf{A}^\ac)\longrightarrow 
{\rm Sel}_{p^\infty}(f/K)^\vee
\]
with finite kernel, where 
\[
{\rm Sel}_{p^\infty}(f/K)^\vee={\rm Hom}_{\bZ_p}({\rm Sel}_{p^\infty}(f/K),\bQ_p/\bZ_p)
\]
is the Pontrjagin dual. The assumption that ${\rm Sel}_{p^\infty}(f/K)$ has $\bZ_p$-corank $1$ thus implies that $\mathfrak{X}^{+,+}(K,\mathbf{A}^\ac)/I^\ac\mathfrak{X}^{+,+}(K,\mathbf{A}^\ac)$ is not $\bZ_p$-torsion.
By part (1) of Theorem~\ref{thm:howard}, this forces $\mathbf{z}^+$ to have nontorsion image in $\mathfrak{Sel}^{+,+}(K,\mathbf{T}^\ac)/I^\ac\mathfrak{Sel}^{+,+}(K,\mathbf{T}^\ac)$. By the natural injection
\[
\mathfrak{Sel}^{+,+}(K,\mathbf{T}^\ac\otimes_{\bZ_p}\bQ_p)/I^{\tt ac}\mathfrak{Sel}^{+,+}(K,\mathbf{T}^\ac\otimes_{\bZ_p}\bQ_p)
\hookrightarrow{\rm Sel}(K,T\otimes_{\bZ_p}\bQ_p)
\]
this shows that $\mathbf{z}^+$ has nonzero image $\mathbf{z}^+_1\in{\rm Sel}(K,T\otimes_{\bZ_p}\bQ_p)$; since
by construction $\mathbf{z}_1^+=z_0[1]$ is the Kummer image of $y_K$, the result follows.
The last claim follows from the Gross--Zagier formula on Shimura curves \cite{YZZ}.
\end{proof}

\subsection{$\Lambda$-adic Gross--Zagier formula}\label{subsec:main-2}

In this section we obtain an analogue of Howard's $\Lambda$-adic Gross--Zagier formula \cite{howard-compmath} for supersingular primes.
Recall the $p$-adic height pairings $\langle\;,\;\rangle_{K_n^\ac}^{\rm cyc}$
on the plus/minus Selmer groups ${\rm Sel}^{\eps,\eps}(K_n^\ac,T)$ introduced in Theorem~\ref{thm:rubin-ht},
and define the $\Lambda_\ac$-adic height pairing
\begin{equation}\label{eq:lambda-ht}
\langle\;,\;\rangle_{K^\ac_\infty}^{\rm cyc}:\mathfrak{Sel}^{\eps,\eps}(K,\Tc)
\otimes_{\Lambda_\ac}\mathfrak{Sel}^{\eps,\eps}(K,\Tc)^\iota
\longrightarrow\bQ_p\otimes_{\bZ_p}\Lambda_\ac\otimes_{\bZ_p}\mathcal{J}
\end{equation}
by the formula
\[
\langle a_\infty,b_\infty\rangle^{\rm cyc}_{K_\infty^\ac}=
\varprojlim_n\sum_{\sigma\in{\rm Gal}(K_n^\ac/K)}\langle a_n,b_n^\sigma\rangle^{\rm cyc}_{K_n^\ac}\cdot\sigma.
\]

Recall that $\mathcal{J}=\mathcal{I}/\mathcal{I}^2$ for $\mathcal{I}$ the augmentation ideal of $\bZ_p[[\Gamma^{\rm cyc}]]$. Upon choosing a topological generator $\gamma_{\rm cyc}\in\Gamma^{\rm cyc}\simeq\mathcal{J}$ the height pairing $(\ref{eq:lambda-ht})$ may be seen as taking values in $\Lambda_\ac\otimes_{\bZ_p}\bQ_p$; we thus define the cyclotomic $\Lambda_\ac$-adic cyclotomi regulator $\mathcal{R}^\eps_{\rm cyc}\subseteq\Lambda_\ac\otimes_{\bZ_p}\bQ_p$
to be the characteristic ideal of the cokernel of $(\gamma_{\rm cyc}-1)^{-1}\circ(\ref{eq:lambda-ht})$.

\begin{thm}\label{thm:3.1.5}
Let $\eps\in\{\pm\}$, and denote by $\mathcal{X}^{\eps}_{\rm tors}$ be the characteristic ideal of 
$\mathfrak{X}^{\eps,\eps}(K,\Ac)_{\rm tors}$. Then
\[
\mathcal{R}_{\rm cyc}^{\eps}\cdot\mathcal{X}_{\rm tors}^{\eps}
=(L_{p,1}^{\eps,\eps}(f/K))
\]
in $(\Lambda_\ac\otimes_{\bZ_p}\bQ_p)/\Lambda_\ac^\times$.
\end{thm}

\begin{proof}
The height formula of Theorem~\ref{thm:rubin-ht} and Lemma~\ref{lem:3.1.1} immediately
yield the equality
\begin{equation}\label{eq:ht}
\mathcal{R}_{\rm cyc}^\eps\cdot
{\rm Char}_{\Lambda_\ac}\biggl(\frac{\mathfrak{Sel}^{\eps,\eps}(K,\Tc)}{\Lambda_\ac\cdot\mathcal{BF}^\eps_\ac}\biggr)
=\mathcal{U}_\ac\cdot\mathcal{H}_\ac^\eps\cdot(L_{p,1}^{\eps,\eps}(f/K))\cdot\eta^\iota
\end{equation}
in $(\Lambda_\ac\otimes_{\bZ_p}\bQ_p)/\Lambda_\ac^\times$, where $\mathcal{U}_\ac=(u_0)$ and $\mathcal{H}_\ac^\eps=(h_0^\eps)$
in the notations of $\S\ref{subsec:rubin}$, and
\[
\eta:={\rm Char}_{\Lambda_\ac}\biggl(\frac{H^1_{\eps}(K_{\pp},\Tc)}{{\rm loc}_{\pp}(\mathfrak{Sel}^{\eps,\eps}(K,\Tc))}\biggr).
\]

By 
Theorem~\ref{3.1}, we see that $\eta$ (and therefore $\eta^\iota$) is nonzero,
while the nonvanishing of $h_0^\eps$ follows from the construction (see Remark~\ref{rem:h-ac}).
On the other hand, from global Poitou--Tate duality
we have the exact sequence
\begin{equation}\label{eq:PT2}
0\longrightarrow\frac{H^1_\eps(K_{\pp},\Tc)}
{{\rm loc}_{\pp}(\mathfrak{Sel}^{\eps,\eps}(K,\Tc))}
\longrightarrow \mathfrak{X}^{{\rm rel},\eps}(K,\Ac)
\longrightarrow \mathfrak{X}^{\eps,\eps}(K,\Ac)\longrightarrow 0.
\end{equation}
Taking $\Lambda_\ac$-torsion in $(\ref{eq:PT2})$ and applying Lemma~\ref{lem:str-rel} we obtain the equality
\begin{equation}\label{eq:str-rel}
{\rm Char}_{\Lambda_\ac}(\mathfrak{X}^{\eps,{\rm str}}(K,\Ac))=\mathcal{X}_{\rm tors}^\eps\cdot\eta^\iota
\end{equation}
in $(\Lambda_\ac\otimes_{\bZ_p}\bQ_p)/\Lambda_\ac^\times$. Combined
with Corollary~\ref{thm:str} (which applies thanks to Theorem~\ref{thm:howard})
and Lemma~\ref{thm:ES}, equation $(\ref{eq:str-rel})$ implies that
\begin{equation}\label{eq:str-rel2}
{\rm Char}_{\Lambda_\ac}\biggl(\frac{\mathfrak{Sel}^{\eps,\eps}(K,\Tc)}{\Lambda_\ac\cdot\mathcal{BF}^\eps_\ac}\biggr)
=\mathcal{H}^\eps_\ac\cdot\mathcal{X}_{\rm tors}^\eps\cdot\eta^\iota
\end{equation}
in $(\Lambda_\ac\otimes_{\bZ_p}\bQ_p)/\Lambda_\ac^\times$. Since the ideal
$\mathcal{U}_\ac$ is invertible by the combination of Theorem~\ref{thm:2-varIMC} and Corollary~\ref{cor:kimIMC},
substituting $(\ref{eq:str-rel2})$ into $(\ref{eq:ht})$, the result follows.
\end{proof}


Note that the formula of Theorem~\ref{thm:3.1.5} has the shape of a $\Lambda_\ac$-adic
analogue of the Birch and Swinnerton-Dyer conjecture. Moreover, by the Heegner point
main conjecture of Theorem~\ref{thm:howard}, it is essentially equivalent to the
following $\Lambda_\ac$-adic Gross--Zagier formula.

\begin{thm}\label{thm:lambda-GZ}
Let $\eps\in\{\pm\}$. Under the hypotheses in Theorem~\ref{thm:howard}, we have the equality
\[
(L_{p,1}^{\eps,\eps}(f/K))=(\langle\mathbf{z}^\eps,\mathbf{z}^\eps\rangle_{K_\infty^\ac}^{\rm cyc})
\]
in $(\Lambda_{\ac}\otimes_{\bZ_p}\bQ_p)/\Lambda_\ac^\times$.
\end{thm}

\begin{proof}
Combining Theorem~\ref{thm:howard} and Theorem~\ref{thm:3.1.5}, we obtain
\begin{align*}
(L_{p,1}^{\eps,\eps}(f/K))&=\mathcal{R}^\eps_{\cyc}\cdot {\rm Char}_{\Lambda_\ac}(\mathfrak{X}^{\eps,\eps}(K,\Ac)_{\tors})\\
&=\mathcal{R}^\eps_{\cyc}
\cdot{\rm Char}_{\Lambda_\ac}\biggl(\frac{\mathfrak{Sel}^{\eps,\eps}(K,\Tc)}{\Lambda_\ac\cdot\mathbf{z}^\eps}\biggr)^2\\
&=(\langle\mathbf{z}^\eps_{},\mathbf{z}^\eps_{}\rangle_{K_\infty^\ac}^{\cyc})
\end{align*}
in $(\Lambda_{\ac}\otimes_{\bZ_p}\bQ_p)/\Lambda_\ac^\times$, where the last equality follows from the definition of $\mathcal{R}^\eps_{\cyc}$.
\end{proof}

\begin{rem}\label{rem:lambda-GZ}
The equality in Theorem~\ref{thm:lambda-GZ} can easily be rewritten in the form of a $\Lambda_{\ac}$-adic Gross--Zagier formula. Indeed, the two elements in $\Lambda_\ac\otimes_{\bZ_p}\bQ_p$ appearing in that formula are invariant under the $\bZ_p$-isomorphism $\iota$ sending $\gamma_{\rm ac}\mapsto\gamma_{\ac}^{-1}$; letting  $y'\in\Lambda_\ac^\times$ denote the ratio of those two elements, we can find (by making an appropriate choice of $\Omega_f^{\tt Hida}$, which is determined up to a $p$-adic unit) an element $y\in 1+(\gamma_\ac-1)\Lambda_\ac$ such that $y=\iota(y)$ and $y^2=y'$. Setting $z_\infty^\eps:=y\mathbf{z}^\eps$, we thus get another generator of $\Lambda_\ac\cdot\mathbf{z}^\eps$ satisfying
\begin{equation}\label{eq:Lambda-GZ}
L^{\eps,\eps}_{p,1}(f/K)=y\iota(y)\langle\mathbf{z}^\eps,\mathbf{z}^\eps\rangle_{K_\infty^\ac}^{\rm cyc}=
\langle z_\infty^\eps,z_\infty^\eps\rangle_{K_\infty^\ac}^{\rm cyc}.\nonumber
\end{equation}
In particular, specialized at the trivial character of $\Lambda_\ac$, this yields a new proof of Kobayashi's $p$-adic Gross--Zagier formula \cite{kobayashi-191} (which 
our formula extends to all characters of $\Lambda_\ac$).
\end{rem}

\bibliographystyle{amsalpha}
\bibliography{Heegner}

\end{document}